\documentclass{amsart}
\usepackage[hidelinks]{hyperref}
\usepackage{cleveref}
\usepackage{amsmath, amssymb}
\usepackage{enumitem}
\usepackage[T1]{fontenc}    

\setlist[enumerate]{label=(\alph*)}

\usepackage{tikz}
\usepackage{string-diagrams}
\usetikzlibrary{backgrounds, calc}
\usetikzlibrary{positioning, arrows.meta}

\usepackage[maxbibnames=10, 
  url=false, 
  date=year,
  isbn=false]{biblatex}

\AtEveryBibitem{\clearlist{language}}

\newcommand{\C}{\mathbb{C}}
\newcommand{\N}{\mathbb{N}}
\newcommand{\Z}{\mathbb{Z}}

\newtheorem{thm}{Theorem}[section]
\newtheorem{proposition}[thm]{Proposition}
\newtheorem{corollary}[thm]{Corollary}
\newtheorem{lemma}[thm]{Lemma}

\theoremstyle{definition}
\newtheorem{remark}[thm]{Remark}
\newtheorem{definition}[thm]{Definition}
\newtheorem{example}[thm]{Example}
\newtheorem{notation}[thm]{Notation}

\newtheorem{introthm}{Theorem}

\theoremstyle{definition}

\title[voltage quantum graphs]{voltage quantum graphs and a Gross--Tucker Theorem for quantum graphs}
\author{Bj\"orn Sch\"afer and Mariusz Tobolski}
\date{\today}

\begin{document}

\begin{abstract}
    A voltage graph is a finite directed graph whose edges are labeled by elements of a finite group $G$. A classical construction of Gross and Tucker associates to every voltage graph with vertex set $V$ a so-called derived graph with vertex set $V \times G$. We generalize their construction to quantum graphs and finite abelian groups. Remarkably, the construction can produce true quantum graphs starting from a classical voltage graph. In this case the obtained quantum graph is quantum isomorphic to a classical graph.
    As a main result we also prove a quantum version of the Gross--Tucker theorem which characterizes precisely which graphs can be written as derived graphs of voltage graphs.
\end{abstract}

\maketitle

\section{Introduction}

By a famous result of Erd\H{o}s and R\'enyi, almost all finite graphs have no symmetries \cite{erdos_asymmetric_1963}. Yet most graphs that one encounters in real life exhibit some symmetry, starting from the cyclic graphs to more complex ones like the Petersen graph. How can one exploit these symmetries? 
In the 1970s, Gross and Tucker developed a method to describe graphs with symmetries in terms of smaller labeled graphs in an insightful way \cite{gross_voltage_1974, gross_topological_1987}. The main object of their theory are so-called \emph{voltage graphs}. A voltage graph is a pair $(\tilde{\Gamma}, \lambda)$ consisting of a finite (not necessarily simple) graph $\tilde{\Gamma}$ together with a labeling $\lambda: E_{\tilde{\Gamma}} \to G$ of its edges by elements of a finite group $G$. Then, they introduced the \emph{derived graph} $\Gamma^\lambda$ as a particular graph on the vertex set $V_{\tilde{\Gamma}} \times G$ where the edges are determined by the labeling $\lambda$. Gross and Tucker's main theorem states that a graph $\Gamma$ is isomorphic to a derived graph $\tilde{\Gamma}^\lambda$ if and only if there is a free action of the group $G$ on $\Gamma$. As an example, Figure~\ref{intro::fig:petersen_graph_as_derived_graph} depicts the Petersen graph $\Gamma$ together with a voltage graph $\tilde{\Gamma}$ such that $\Gamma \cong \tilde{\Gamma}^\lambda$.
\begin{figure}
    \centering
    \begin{tikzpicture}[scale=1.5]
    \begin{scope}[every node/.style={circle, fill=black, inner sep=1.5pt}]
        \foreach \i in {1,...,5} {
          \node (o\i) at ({72*\i+18}:1) {};
        }
        \foreach \i in {1,...,5} {
          \node (i\i) at ({72*\i+18}:0.4) {};
        }
        \foreach \i [evaluate=\i as \j using {int(mod(\i,5)+1)}] in {1,...,5} {
          \draw (o\i) -- (o\j);
        }
        \foreach \i [evaluate=\i as \j using {int(mod(\i+1,5)+1)}] in {1,...,5} {
          \draw (i\i) -- (i\j);
        }
        \foreach \i in {1,...,5} {
          \draw (o\i) -- (i\i);
        }
    \end{scope}
        
    \begin{scope}[xshift=3cm, vertex/.style={circle, fill=black, inner sep=1.5pt}]
        \node[vertex] (L) at (0,0) {};
        \node[vertex] (R) at (1,0) {};
        \draw[->, looseness=20]
          (L) to[out=130,in=230] node[left] {$1_{\Z_5}$} (L);
        \draw[->, looseness=20]
          (R) to[out=50,in=310] node[right] {$0_{\Z_5}$} (R);
        \draw[->]
          (L) -- node[above] {$2_{\Z_5}$} (R);
        \end{scope}
    \end{tikzpicture}
    \caption{The Petersen graph is the derived graph of a voltage graph over the group $\Z_5$.}
    \label{intro::fig:petersen_graph_as_derived_graph}
\end{figure}
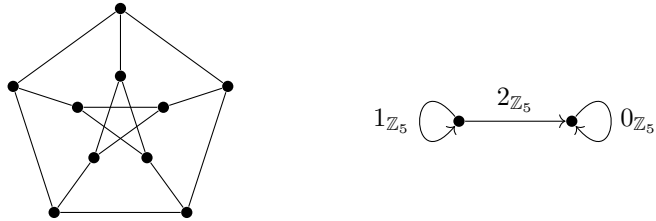
Voltage and derived graphs have found applications in a number of settings, e.g. for generating graph coverings \cite{gross_coverings_1977}, for constructing large graphs with small degree and diameter \cite{brankovic_large_1998} or for lifting graph automorphisms \cite{malnic_lifting_2000}. Kumjian and Pask used voltage graphs to decompose graph C*-algebras into crossed products up to stable isomorphism \cite{kumjian_c-algebras_1999}. 

In this article we generalize derived graphs and the Gross--Tucker theorem to \emph{quantum graphs}. Quantum graphs are a recent generalization of classical graphs where the (finite) vertex set is replaced with an arbitrary finite-dimensional C*-algebra. By Gelfand duality a finite set of $n$ points corresponds to the C*-algebra $\C^n$, and in this spirit any finite-dimensional C*-algebra is considered as a (finite) quantum set. For a number of reasons, though, a quantum set is indeed a pair $(B, \psi)$ of a finite-dimensional C*-algebra and a particular functional $\psi$.
For instance, a classification of (undirected) quantum graphs on the quantum set $(M_2, 2 \mathrm{Tr})$ can be found in \cite{gromada_examples_2022, matsuda_classification_2022}, where $M_2$ is the C*-algebra of $2 \times 2$ complex matrices.

Quantum graphs were introduced in the context of quantum channels by Duan, Severini and Winter in 2012 \cite{duan_zero-error_2013} (under the name of \emph{noncommutative graphs}). Independently, Weaver had developed a theory of quantum relations which also encompasses quantum graphs \cite{weaver_quantum_2012,weaver_quantum_graphs_2015}. Later work of Musto, Reutter and Verdon \cite{musto_compositional_2018} placed quantum graphs in a categorical framework which includes quantum versions of sets, (bijective) functions and more. Importantly, there are three different ways to define a quantum graph, but all definitions lead to equivalent notions of quantum graphs, see e.g. \cite{musto_compositional_2018, daws_quantum_2024, wasilewski_quantum_2024}. Other applications of quantum graphs include the recent definition of quantum Cayley graphs by Wasilewski \cite{wasilewski_quantum_2024} and the description of all finite quantum groups as quantum automorphism groups of quantum graphs \cite{brannan_quantum_2025}. Further, quantum graphs are used to define quantum Cuntz--Krieger algebras in \cite{brannan_quantum_CuntzKrieger_2022, brannan_quantum_edge_2022}.

In this paper, we introduce the notion of \emph{voltage quantum graphs} and their \emph{derived quantum graphs} (Definition \ref{der::def:derived_quantum_graph}).
Intuitively, a voltage quantum graph is a quantum multi-graph with an edge labeling by elements of a finite abelian group $G$. Classically, the derived graph of a voltage graph is defined on the Cartesian product of the vertex set of the voltage graph with the group $G$. In the quantum setting, we have more flexibility: A voltage quantum graph is a $G$-labeled family of quantum adjacency matrices $\tilde{A} = (\tilde{A}_g)_{g \in G}$ on a quantum set $(\tilde{B}, \tilde{\psi})$ \textit{with respect to} an action $\hat{\alpha}$ of the dual group $\hat{G}$ on $\tilde{A}$. Then, the derived quantum graph
\begin{align*}
    A := \tilde{A} \rtimes_{\hat{\alpha}} \hat{G}
\end{align*}
is defined on the \emph{crossed product quantum set} (Definition \ref{crossed::def:crossed_product_quantum_set})
\begin{align*}
    (\tilde{B}, \tilde{\psi}) \rtimes_{\hat{\alpha}} \hat{G} := (\tilde{B} \rtimes_{\hat{\alpha}} \hat{G}, \psi_\rtimes),
\end{align*}
where $\psi_\rtimes$ is suitably defined.
To our knowledge, crossed product quantum sets have not been considered in the literature before. Note that the action $\hat{\alpha}$ has no equivalent in Gross and Tucker's theory. Their voltage graphs are recovered in the special case where $\tilde{A}$ is a family of classical adjacency matrices and $\hat{\alpha}$ is the trivial action. To obtain more general results for non-abelian groups, we intend to investigate voltage quantum graphs with respect to \emph{quantum} groups in a future work.

Quantum graphs on quantum sets of the form $(\C^n, \psi_n)$ (for suitable $\psi_n$) correspond to classical graphs on $n$ vertices. It is a remarkable feature of our construction that one can start with a classical voltage graph $\tilde{A}$ and obtain a quantum graph which is defined on a true quantum set, i.e. one that does not correspond to a classical set. For instance, let $\tilde{A}$ be a classical voltage graph on $2$ vertices. If $\hat{\alpha}$ is the action of $\hat{\Z}_2$ which swaps the two vertices, then the crossed product quantum set is
\begin{align*}
    (\C^2, \psi_2) \rtimes_{\hat{\alpha}} \hat{\Z}_2 \cong (M_2, 2 \mathrm{Tr}).
\end{align*}
Accordingly, the derived quantum graph $\tilde{A} \rtimes_{\hat{\alpha}} \hat{\Z}_2$ is a quantum graph on the quantum set $(M_2, 2 \mathrm{Tr})$. In fact, by \cite{matsuda_classification_2022, gromada_examples_2022} there are only four non-isomorphic (undirected, loopfree) quantum graphs on $(M_2, 2 \mathrm{Tr})$, and they are all obtained in this way (Proposition \ref{ex::prop:qgraphs_on_M2_as_derived_qgraphs}).

Interestingly, all these four quantum graphs on $(M_2, 2 \mathrm{Tr})$ are \emph{quantum isomorphic} to classical graphs on four vertices \cite{matsuda_classification_2022}. Quantum isomorphism is a weaker notion than classical isomorphism, and it is defined in the very same spirit as quantum sets and quantum graphs. Quantum isomorphisms of classical graphs have been extensively studied \cite{atserias_quantum_2019, mancinska_quantum_2019}.
We recover the quantum isomorphism between quantum graphs on $(M_2, 2 \mathrm{Tr})$ and $(\C^4, \psi_4)$ as a special case of a much more general phenomenon.

\begin{introthm}[Theorem \ref{qiso::thm:quantum_isomorphism_between_derived_graphs}]
    If $\tilde{A} = (\tilde{A}_g)_{g \in G}$ is a voltage quantum graph on $(\tilde{B}, \tilde{\psi})$ with respect to the action $\hat{\alpha}$, then we have a quantum isomorphism
    \begin{align*}
        \tilde{A} \rtimes_{\hat{\alpha}} \hat{G} \cong_q \tilde{A} \rtimes_{\mathrm{triv}} \hat{G},
    \end{align*}
    where we denote by $\mathrm{triv}$ the trivial action of $\hat{G}$ on $(\tilde{B}, \tilde{\psi})$.
\end{introthm}

A main upshot of the previous theorem is a method to construct ``quantum twins'' of a large number of classical graphs. By a quantum twin we mean a quantum graph that is quantum isomorphic to a given classical graph $\Gamma$. This method proceeds in three steps:
\begin{enumerate}
    \item Given a free action of a finite abelian group $G$ on $\Gamma$, use the Gross--Tucker theorem to find a (classical) voltage graph $(\tilde{\Gamma}, \lambda)$ such that $\Gamma \cong \tilde{\Gamma}^\lambda$.
    \item Find a non-trivial action $\hat{\alpha}$ of $\hat{G}$ on $(\tilde{\Gamma}, \lambda)$.
    \item Conclude with Theorem \ref{qiso::thm:quantum_isomorphism_between_derived_graphs} that $(\tilde{\Gamma}, \lambda) \rtimes_{\hat{\alpha}} \hat{G}$ is a quantum graph that is quantum isomorphic to $\Gamma$.
\end{enumerate}
Here we identify a classical voltage graph $(\tilde{\Gamma}, \lambda)$ with a voltage quantum graph $\tilde{A} = (\tilde{A}_g)_{g \in G}$ in a natural way (see Proposition \ref{der::prop:classical_voltage_graphs_as_quantum_voltage_graphs}).

Finally, the main result of this work is a generalization of the Gross--Tucker theorem to quantum graphs: Classically, a graph $\Gamma$ is isomorphic to a derived graph $\tilde{\Gamma}^\lambda$ if and only if there is a free action of $G$ on $\Gamma$ such that $\tilde{\Gamma} \cong \Gamma / G$, i.e. the (unlabeled) voltage graph is isomorphic to the quotient graph. The first problem in generalizing this theorem is to find an analogue of a free action on a quantum set. Here, the answer is given by a theorem of Landstad \cite{landstad_duality_1979} which precisely characterizes when a given C*-algebra $B$ is isomorphic to a crossed product algebra $\tilde{B} \rtimes_{\hat{\alpha}} \hat{G}$. Note that the conditions in Landstad's theorem are stronger than mere freeness of the action which has been generalized to the setting of C*-algebras in various ways~\cite{phillips_freeness_1987, rieffel_proper_1990, baum_decommer_hajac_free_2017}.

\begin{introthm}[Theorem \ref{gtt::thm:quantum_gross_tucker_theorem}]
    Let $A: B \to B$ be a quantum graph on the quantum set $(B, \psi)$ and let $\alpha: G \to \mathrm{Aut}(A)$ be an action of a finite abelian group $G$ on $A$ which has the following properties:
    \begin{enumerate}
        \item There is a representation $(u_\chi)_{\chi \in \hat{G}} \subset B$ of the dual group $\hat{G}$ such that
        \begin{align*}
            \alpha_g(u_\chi) = \chi(g) u_\chi
        \end{align*}
        holds for all $g \in G$ and $\chi \in \hat{G}$.
        \item One has 
        \begin{align*}
            \psi \mathrm{Ad}(u_\chi) = \psi 
            \quad \text{ as well as } \quad
            \mathrm{Ad}(u_\chi) A = A \mathrm{Ad}(u_\chi)
        \end{align*}
        for all $\chi \in \hat{G}$.
    \end{enumerate}
    Then there is a voltage quantum graph $\tilde{A}$ on $(B^\alpha, \tilde{\psi})$ such that
    \begin{align*}
        A \cong \tilde{A} \rtimes_{\hat{\alpha}} \hat{G},
    \end{align*}
    where $\hat{\alpha}: \hat{G} \to \mathrm{Aut}(B^\alpha, \tilde{\psi})$ is given by $\hat{\alpha}_\chi(b) = u_\chi b u_\chi^*$ for all $b \in B^\alpha$ and $\chi \in \hat{G}$, and $\tilde{\psi} = \frac{1}{|G|} \psi|_{B^\alpha}$.
\end{introthm}

To conclude, let us discuss the connection between classical derived graphs and graph C*-algebras in more detail. Graph C*-algebras have been introduced by Kumjian, Pask and Raeburn in \cite{kumjian_graphs_1997} as a generalization of Cuntz--Krieger algebras \cite{cuntz_class_1980}. 
A graph C*-algebra $C^\ast(\Gamma)$ is associated to a directed, non-simple graph $\Gamma$ so that many properties of the C*-algebra can be studied in terms of combinatorial properties of the graph $\Gamma$, see e.g. \cite{raeburn_graph_2005}. 
According to a theorem of Kumjian and Pask \cite{kumjian_c-algebras_1999} the graph C*-algebra of a derived graph is stably isomorphic to a crossed product C*-algebra, i.e. we have a stable isomorphism of the form
\begin{align*}
    C^\ast(\Gamma^\lambda) \cong_{\mathrm{stably}} C^\ast(\Gamma) \rtimes G.
\end{align*} 
Recently quantum Cuntz-Krieger C*-algebras have been introduced as a generalization of particular graph C*-algebras to quantum graphs \cite{brannan_quantum_CuntzKrieger_2022, brannan_quantum_edge_2022}. We plan to investigate if Kumjian and Pask's result can be extended to the setting of quantum Cuntz-Krieger algebras. Indeed, this has been an initial motivation for the present work.

\subsection{Outline}

In the next \textbf{Section \ref{sec::preliminaries}}, we introduce some notation about quantum graphs and crossed product C*-algebras. In \textbf{Section \ref{sec::crossed_prod_qsets}}, we discuss how an action on a quantum set gives rise to a crossed product quantum set. Then, in \textbf{Section \ref{sec::qvoltage_and_derived_graphs}}, we are ready to define voltage quantum graphs and derived quantum graphs. We show that derived quantum graphs with respect to different actions of the dual group are quantum isomorphic in \textbf{Section \ref{sec::quantum_isomorphisms}}, before we prove a Gross--Tucker theorem for derived quantum graphs in \textbf{Section \ref{sec::qgross_tucker}}. The paper ends with some examples in \textbf{Section \ref{sec::examples}}.

\subsection{Acknowledgments}

This work is part of the international project ``Graph Algebras'' supported by EU framework HORIZON-MSCA-2021-SE-01, grant number 101086394. The international project was also co-financed by the Polish Ministry of Science and Higher Education through the grant W24/HE/2023 (Mariusz Tobolski).
The authors would like to thank the organizers of  the 20th Workshop on ``Noncommutative Probability, Operator Algebras and Related Topics, with Applications'' held at the Mathematical Research and Conference Center in B\k{e}dlewo, Poland, in July 2025 where the work on this project started. BS also acknowledges the SFB-TRR 195 and thanks his PhD supervisor Moritz Weber for guidance and support. This work is part of the PhD thesis of BS at Saarland University. BS also thanks Wroc\l aw university for their hospitality, and MT would like to thank Saarland University for great hospitality and working conditions.

\section{Preliminaries}
\label{sec::preliminaries}

Generally, $B$ denotes a C*-algebra and $\psi$ denotes a functional on it. If $B$ is equipped with a faithful positive functional $\psi$, this pair gives rise to a Hilbert space $L^2(B)$ via the GNS-construction. In particular, if $B$ is finite-dimensional, then $L^2(B) = B$ as a set and the inner product is given by $\langle a, b \rangle_\psi = \psi(a^\ast b)$. In our case, $B$ will usually be finite-dimensional, and it will be clear which functional it is equipped with. Then we do not write $L^2(B)$ but simply regard $B$ as a Hilbert space. We denote the C*-algebras of complex $n \times n$ matrices by $M_n$.

Further, if $\hat{G}$ is the dual of a finite abelian group $G$ and $\hat{\alpha}: \hat{G} \to \mathrm{Aut}(\tilde{B})$ is an action of $\hat{G}$ on a C*-algebra $\tilde{B}$, then we denote by $\tilde{B} \rtimes_{\hat{\alpha}} \hat{G}$ the crossed product C*-algebra. We will consistently write $\tilde{B}$ for the smaller C*-algebra on which $\hat{G}$ acts, and we write $B$ for the crossed product C*-algebra. Accordingly, a functional on $\tilde{B}$ is usually denoted by $\tilde{\psi}$, and a functional on $B$ is denoted by $\psi$. Sometimes, we also write $\psi_\rtimes$ for $\psi$ to emphasize that it is defined on a crossed product C*-algebra. The unit of a group is always denoted by $1$ where the group is clear from the context. If $\chi, \xi$ are two elements of the dual group, then $\delta_{\chi=\xi}$ is the Kronecker delta.
If $u \in B$ is a unitary in a C*-algebra, then we denote by $\mathrm{Ad}(u)$ the automorphism of $B$ given by $b \mapsto u b u^\ast$ for all $b \in B$.

\subsection{Quantum sets}

Quantum sets have been introduced by Musto, Reutter and Verdon in \cite{musto_compositional_2018} together with an extensive categorical framework encompassing quantum functions and quantum graphs. In short, quantum sets follow the spirit of Gelfand duality where a classical finite set $X$ is identified with the commutative C*-algebra $C(X)$. Generalizing this, any finite-dimensional C*-algebra $B$ gives rise to a quantum set. However, for a number of reasons, it is important to also equip the C*-algebra $B$ with a particular functional $\psi$. Then one arrives at the following definition which is essentially due to Musto, Reutter and Verdon \cite{musto_compositional_2018}.

\begin{definition}
    A quantum set $(B, \psi)$ is a pair consisting of a finite-dimensional C*-algebra $B$ and a faithful positive functional $\psi: B \to \C$ such that
    \begin{align*}
        m m^\ast = \mathrm{Id}_B
    \end{align*}
    holds for the multiplication map $m: B \otimes B \to B$ and its adjoint $m^\ast$ given by $\langle m(a \otimes b), c \rangle_\psi = \langle a \otimes b, m^\ast(c) \rangle_{\psi \otimes \psi}$ for all $a, b, c \in B$.
    We call any functional with this property a \emph{quantum set functional} on $B$.
\end{definition}

We think of $B$ as a C*-algebra that is at the same time a Hilbert space with the inner product $\langle \cdot, \cdot \rangle_{\psi}$ induced by the functional $\psi$. Maps involving $B$ are generally considered as linear maps with additional properties; thus, we will generally denote the concatenation of two $\ast$-homomorphisms $\varphi_1, \varphi_2: B \to B$ by $\varphi_1 \varphi_2$ instead of $\varphi_1 \circ \varphi_2$. Elements $b \in B$ are identified with the linear map $\C \to B, \lambda \mapsto \lambda b$, and therefore have an adjoint denoted by $b^\dagger$. This adjoint is given by 
$$
    b^\dagger: B \to \C,\qquad c \mapsto \langle b, c \rangle_\psi.
$$

\begin{remark}
    An alternative and widely-used definition of a quantum set asks that the functional $\psi$ be a $\delta$-form for some $\delta > 0$ which means that
    \begin{align*}
        m^\ast m = \delta^2 \mathrm{Id}_{B \otimes B}
    \end{align*}
    holds. The difference between this condition and the above definition is merely a scalar factor: Every quantum set functional is a scaled $\delta$-form and vice versa. See \cite{wasilewski_quantum_2024} for more details.
\end{remark}

The functional $\psi$ is called the \emph{counit} and $m^\ast$ is called the \emph{comultiplication} on $B$. We collect some of their basic properties below which are well-known, see e.g. \cite{musto_compositional_2018}.

\begin{proposition}
    \label{pre::prop:properties_of_counit_and_comultiplication}
    The maps $\psi$ and $m^\ast$ have the following properties:
    \begin{enumerate}
        \item The adjoint $\psi^\ast$ is the unit in the sense that $\psi^\ast: \C \to B, \lambda \mapsto \lambda 1_B$.
        \item The comultiplication is co-associative, i.e.
        \begin{align*}
            (\mathrm{Id}_B \otimes m^\ast)m^\ast = (m^\ast \otimes \mathrm{Id}_B)m^\ast.
        \end{align*}
        \item The comultiplication satisfies
        \begin{align*}
            (\psi \otimes \mathrm{Id}_B)m^\ast = \mathrm{Id}_B = (\mathrm{Id}_B \otimes \psi)m^\ast.
        \end{align*}
    \end{enumerate}
\end{proposition}

\begin{example}
    \label{pre::ex:examples_of_quantum_sets}
    \begin{enumerate}
        \item 
            Every classical finite set $X$ gives rise to a quantum set $(C(X), \psi_X)$, where $\psi_X: C(X) \to \C$ is the functional given by $\psi_X(f) = \sum_{x \in X} f(x)$ for all $f \in C(X)$. Indeed, in this case $m^\ast(e_x) = e_x \otimes e_x$ holds for all $x \in X$, and thus $m m^\ast(e_x) = m(e_x \otimes e_x) = e_x$. Here $e_x$ denotes the indicator function at $x \in X$.
        \item For every matrix algebra $M_n$ with $n \in \N$, a quantum set is obtained by equipping $M_n$ with the functional $\psi := n \mathrm{Tr}$, where $\mathrm{Tr}$ is the un-normalized trace on $M_n$. In fact, this is the unique \emph{tracial} quantum set functional on $M_n$.
    \end{enumerate}
\end{example}

\subsection{Quantum graphs}

A simple graph on a classical finite set $\{1, \dots, n\}$ is given by its adjacency matrix $A \in \{0,1\}^{n \times n}$. One can interpret the adjacency matrix as a linear map $A: \C^n \to \C^n$ with the special property that it is invariant under entrywise multiplication with itself. This can be generalized to quantum sets as follows.

\begin{definition}
    \begin{enumerate}
        \item A \emph{quantum adjacency matrix} on a quantum set $(B, \psi)$ is a linear map $A: B \to B$ which is Schur-idempotent and $\ast$-preserving, by which we mean that it satisfies
        \begin{align*}
            m(A \otimes A)m^\ast = A 
            \quad \text{ and } \quad
            A(b^\ast) = A(b)^\ast \quad \text{ for all } b \in B.
        \end{align*}

        \item A \emph{quantum graph} on $(B, \psi)$ is given by a quantum adjacency matrix on this quantum set. 
    \end{enumerate}
\end{definition}

This definition is essentially due to Musto, Reutter and Verdon \cite{musto_compositional_2018}. By Matsuda \cite{matsuda_classification_2022} the second condition, $A(b^\ast) = A(b)^\ast$, is equivalent to asking that $A$ is completely positive. For a closer discussion of quantum graphs see e.g. \cite{daws_quantum_2024, wasilewski_quantum_2024}.

Throughout this article, we will identify quantum graphs and quantum adjacency matrices, i.e. we will call a matrix $A: B \to B$ with the necessary properties a quantum adjacency matrix as well as a quantum graph.

\begin{remark}
    Let us remark that quantum graphs can also be defined in different ways. First, a classical graph is given by a subset of $V \times V$, i.e. by its edge set. In a quantum generalization, one says that a quantum graph is given by a projection $P$ in $B \otimes B^{\mathrm{op}}$ where the latter is the opposite C*algebra. Further, a classical graph can be described by the space $\mathrm{span}(E_{ij} : i \text{ is adjacent to } j) \subset M_n$, and this can be generalized to quantum sets as well. Crucially, these different definitions of quantum graphs are equivalent, see e.g. \cite{musto_compositional_2018, daws_quantum_2024,wasilewski_quantum_2024}. This shows that the above definition is indeed sensible.
\end{remark}

\begin{example}
    \label{pre::ex:classical_graphs_as_quantum_graphs}
    For all $n \in \N$ the pair $(\C^n, \psi_n)$ with $\psi_n(e_i) = 1$ for all $i \leq n$ is a quantum set, see Example \ref{pre::ex:examples_of_quantum_sets}. In this case, the operation $m(\cdot \otimes \cdot)m^\ast$ produces exactly the entrywise product of two matrices. Thus, a matrix $A \in M_n$ is the classical adjacency matrix of a (simple, directed) graph on $n$ vertices if and only if $A$ is a quantum adjacency matrix on the quantum set $(\C^n, \psi_n)$. 
    
    We will use this identification later on to show that our definitions and theorems align with the classical theory of voltage graphs and derived graphs. Occasionally, we also write $(\C^V, \psi_V)$ for the quantum set corresponding to a finite set $V$.
\end{example}

\subsection{Definition of quantum isomorphisms}

Let us recall isomorphisms and quantum isomorphisms between quantum sets and quantum graphs. We mainly follow the theory of Musto, Reutter and Verdon \cite{musto_compositional_2018} with the difference that we do not take a categorical perspective on quantum sets and quantum graphs but work with explicit C*-algebras. 

\begin{definition}
    \label{pre::def:qisomoprhisms_of_qsets_and_qgraphs}
    Let $(B_1, \psi_1)$ and $(B_2, \psi_2)$ be quantum sets and let $A_1: B_1 \to B_1$ and $A_2: B_2 \to B_2$ be quantum graphs on $(B_1, \psi_1)$ and $(B_2, \psi_2)$, respectively.
    \begin{enumerate}
        \item The two quantum sets are \emph{isomorphic} if there is a $*$-isomorphism $\varphi: B_1 \to B_2$ such that $\psi_2 \varphi = \psi_1$. We denote by $\mathrm{Aut}(B, \psi)$ the group of quantum set automorphisms of a quantum set $(B, \psi)$.
        \item The two quantum graphs are \emph{isomorphic} if there is an isomorphism $\varphi: B_1 \to B_2$ of the underlying quantum sets such that $\varphi A_1 = A_2 \varphi$. We denote by $\mathrm{Aut}(A)$ the group of quantum graph automorphisms of a quantum graph $A$.
        \item The two quantum sets are \emph{quantum isomorphic}, written $(B_1, \psi_1) \cong_q (B_2, \psi_2)$, if there is a finite-dimensional Hilbert space $\mathcal{H}$ and a unital $\ast$-homomorphism  
        \begin{align*}
            \rho: B_1 \to B_2 \otimes B(\mathcal{H}),
        \end{align*}
        such that the map
        \begin{align*}
            p: L^2(B_1) \otimes \mathcal{H} \to L^2(B_2) \otimes \mathcal{H}, \quad a \otimes \xi \mapsto \rho(a)(1 \otimes \xi)
        \end{align*}
        is unitary as an operator between Hilbert spaces.
        \item The two quantum graphs are \emph{quantum isomorphic}, written $A_1 \cong_q A_2$, if there is a quantum isomorphism $\rho: B_1 \to B_2 \otimes B(\mathcal{H})$ between the underlying quantum sets such that
        \begin{align*}
            \rho A_1 = (A_2 \otimes \mathrm{Id}_{B(\mathcal{H})}) \rho
        \end{align*}
        holds.
    \end{enumerate}
\end{definition}

\begin{remark}
    The previous definition is a rephrased version of the definition given by Musto, Reutter and Verdon in \cite{musto_compositional_2018}. They define a quantum function from $(B_1, \psi_1)$ to $(B_2, \psi_2)$ as a linear map
    \begin{align*}
        P: \mathcal{H} \otimes B_1 \to B_2 \otimes \mathcal{H},
    \end{align*}
    which satisfies certain properties. Their map $P$ corresponds to our map $p$ after swapping the tensor legs in the domain, and their properties on $P$ correspond to asking that the associated map $\rho$ is a $\ast$-homomorphism.

    Moreover, Musto, Reutter and Verdon have additional properties which $P$ must satisfy in order to be called \emph{bijective}. In \cite[Theorem 4.8]{musto_compositional_2018} several equivalent characterizations of bijective quantum functions are given. One characterization asks that the map $P$ is unitary, and that is the property that we use in our definition of quantum isomorphisms.

    Let us finally mention that the above formulation of quantum isomorphisms follows mainly Brannan et al. \cite{brannan_bigalois_2020} with the only difference that we adopt a basis-free perspective.
\end{remark}

\begin{notation}
    \label{pre::notation:quantum_functions}
    It is convenient to treat a quantum isomorphism $\rho: B_1 \to B_2 \otimes B(\mathcal{H})$ as if it were a map from $B_1$ to $B_2$. For that we use the following notational conventions.
    \begin{enumerate}
        \item If $T: B_2 \to B_3$ is a linear map, then we write
        \begin{align*}
            T \rho
            \text{ for the map }
            (T \otimes \mathrm{Id}_{\mathcal{H}}) \rho: B_1 \to B_3 \otimes B(\mathcal{H}).
        \end{align*}
        An equation of the form 
        \begin{align*}
            T \rho = S
        \end{align*}
        for some $S: B_1 \to B_3$ is, accordingly, to be understood as
        \begin{align*}
            (T \otimes \mathrm{Id}_{\mathcal{H}}) \rho = S \otimes \mathrm{Id}_{\mathcal{H}}.
        \end{align*}
        \item Further, we understand the tensor product $\rho \otimes \rho$ in a special way. Indeed, for any $x, y \in B_1$ the elements $\rho(x)$ and $\rho(y)$ can be interpreted as matrices with values in $B_2$ since $\mathcal{H}$ is finite-dimensional. Let us embed the entries of $\rho(x)$ in the left tensor leg of $B_2 \otimes B_2$, and the entries of $\rho(y)$ in the right tensor leg of $B_2 \otimes B_2$. Then, $\rho(x)$ and $\rho(y)$ are two matrices with values in $B_2 \otimes B_2$. These two matrices can be multiplied according to the usual matrix multiplication rules. Then we obtain a new matrix of the same size as before with entries in $B_2 \otimes B_2$. It is this matrix that we denote by $(\rho \otimes \rho)(x \otimes y)$. More precisely, we write
        \begin{align*}
            \rho \otimes \rho 
            \text{ for the map }
            (\mathrm{Id}_{B_2} \otimes \mathrm{Id}_{B_2} \otimes m_{B(\mathcal{H})}) \Sigma_{1, 3, 2, 4} (\rho \otimes \rho),
        \end{align*}
        where $\Sigma: B_2 \otimes B(\mathcal{H}) \otimes B_2 \otimes B(\mathcal{H}) \to B_2 \otimes B_2 \otimes B(\mathcal{H}) \otimes B(\mathcal{H})$ swaps the two tensor factors in the middle and $m_{B(\mathcal{H})}: B(\mathcal{H}) \otimes B(\mathcal{H}) \to B(\mathcal{H})$ is the multiplication map on $B(\mathcal{H})$.
        \item Accordingly, if $T: B_2 \to B_3$ and $S: B_2 \to B_4$ are linear maps, then we write
        \begin{align*}
            &(S \otimes T)(\rho \otimes \rho)
            \text{ for the map } \\
            &\qquad (\mathrm{Id}_{B_3} \otimes \mathrm{Id}_{B_4} \otimes m_{B(\mathcal{H})}) (S \otimes T \otimes \mathrm{Id}_{\mathcal{H} \otimes \mathcal{H}}) \Sigma_{1, 3, 2, 4} (\rho \otimes \rho).
        \end{align*}
    \end{enumerate}
    Let us stress, that these notational conventions are in accordance with the theory of Musto, Reutter and Verdon for computing with quantum functions and quantum isomorphisms.
\end{notation}

The following proposition collects some elementary properties of quantum isomorphisms. In fact, these properties can be used for defining quantum isomorphisms as well, see \cite[Theorem 4.8]{musto_compositional_2018}.

\begin{proposition}
    \label{pre::prop:properties_of_qisomorphisms}
    Let $\rho: B_1 \to B_2 \otimes B(\mathcal{H})$ be a quantum isomorphism between the quantum sets $(B_1, \psi_1)$ and $(B_2, \psi_2)$. Then the following properties hold:
    \begin{enumerate}
        \item $m_2(\rho \otimes \rho) = \rho m_1$,
        \item $\psi_2 \rho = \psi_1$,
        \item $(\rho \otimes \rho) m_1^\ast = m_2^\ast \rho$,
        \item $\rho \psi_1^\ast = \psi_2^\ast$,
    \end{enumerate}
    where $m_1$ ($m_2$) is the multiplication map on $B_1$ ($B_2$).
\end{proposition}

\begin{proof}
    The first equation says that $\rho$ is multiplicative, and this must be true since $\rho$ is a $\ast$-homomorphism. The second equation follows from the fact that $p$ is unitary, see \cite[Lemma 4.2]{brannan_bigalois_2020}. The last two equations are essentially the properties that Musto, Reutter and Verdon require for a quantum function to be bijective, see \cite[Theorem 4.8]{musto_compositional_2018}.
\end{proof}

\subsection{Crossed products}

We refer the reader to Blackadar's book \cite{blackadar_operator_2006} for a general reference on crossed product C*-algebras. If $\tilde{B}$ is a finite-dimensional C*-algebra and $\hat{\alpha}: \hat{G} \to \mathrm{Aut}(\tilde{B})$ is a group action of a finite group on $B$, then the crossed product $B \rtimes_{\hat{\alpha}} \hat{G}$ is again a finite-dimensional C*-algebra. This is the situation which we will encounter throughout the rest of this article. The fact that we always consider actions of duals of finite abelian groups is due to notational cosmetics for later results. Indeed, any finite abelian group is the dual of another finite abelian group.
We will consistently use the following notation.

\begin{notation}
    \label{pre::notation:crossed_product}
    Let $(\tilde{B}, \tilde{\psi})$ be a quantum set and let $\hat{\alpha}: \hat{G} \to \mathrm{Aut}(\tilde{B}, \tilde{\psi})$ be an action of the dual of a finite abelian group $G$. Further, let
    \begin{align*}
        B := \tilde{B} \rtimes_{\hat{\alpha}} \hat{G}
    \end{align*}
    be the associated crossed product C*-algebra. There is a representation 
    $$
    (u_\chi)_{\chi \in \hat{G}} \subset B
    \quad \text{ with } \quad
    \hat{\alpha}_\chi(b) = u_\chi b u_\chi^\ast
    \text{ for all } b \in \tilde{B}, \chi \in \hat{G},
    $$ 
    such that every element of $B$ can be written uniquely as
    \begin{align*}
        b = \sum_{\chi \in \hat{G}} b_\chi u_\chi
        \qquad \text{ for suitable } b_\chi \in \tilde{B}.
    \end{align*}
    Further, multiplication and involution on $B$ are determined by
    \begin{align*}
        (b_\chi u_\chi) (c_\xi u_\xi) &= b_\chi \hat{\alpha}_\chi(c_\xi) u_{\chi \xi}, \\
        (b_\chi u_\chi)^\ast &= \hat{\alpha}_{\chi^{-1}}(b_\chi^\ast) u_{\chi^{-1}}.
    \end{align*}
    Finally, there is a linear map $E_\rtimes: B \to \tilde{B}$ given by
    \begin{align*}
        E_\rtimes\left( \sum_{\chi \in \hat{G}} b_\chi u_\chi \right) = n b_1,
    \end{align*}
    where $n := |G| = |\hat{G}|$. This is a scaled version of the natural conditional expectation from $B$ onto $\tilde{B}$. The scaling is chosen so that it fits with the quantum set functionals that we will associate with $B$ and $\tilde{B}$.
\end{notation}

\section{Crossed product quantum sets}
\label{sec::crossed_prod_qsets}

Assume we have given a finite-dimensional C*-algebra $\tilde{B}$ together with a dual group action $\hat{\alpha}: \hat{G} \to \mathrm{Aut}(\tilde{B})$ where $G$ is a finite abelian group. The crossed product construction is an operation
\begin{align*}
    \tilde{B} \mapsto B := \tilde{B} \rtimes_{\hat{\alpha}} \hat{G},
\end{align*}
which produces a new C*-algebra $B = \tilde{B} \rtimes_{\hat{\alpha}} \hat{G}$, which is called the \emph{crossed product C*-algebra}, from the pair $(B, \hat{\alpha})$.

What happens if $\tilde{B}$ is equipped with a functional $\tilde{\psi}$ such that $(\tilde{B}, \tilde{\psi})$ is a quantum set? In this section, we discuss how one can construct a functional $\psi_\rtimes$ on the crossed product C*-algebra $\tilde{B} \rtimes_{\hat{\alpha}} \hat{G}$ such that $(\tilde{B} \rtimes_{\hat{\alpha}} \hat{G}, \psi_\rtimes)$ is again a quantum set. This gives us an operation
\begin{align*}
    (\tilde{B}, \tilde{\psi}) \mapsto (\tilde{B}, \tilde{\psi}) \rtimes_{\hat{\alpha}} \hat{G} =: (B, \psi_\rtimes),
\end{align*}
which produces a new quantum set $(B, \psi_\rtimes)$, which we call the \emph{crossed product quantum set}, from a quantum set $(\tilde{B}, \tilde{\psi})$ together with a dual group action $\hat{\alpha}$.

\begin{definition}
    \label{crossed::def:crossed_product_quantum_set}
    Let $(\tilde{B}, \tilde{\psi})$ be a quantum set and let $\hat{\alpha}: \hat{G} \to \mathrm{Aut}(\tilde{B}, \tilde{\psi})$ be an action of the dual of a finite abelian group $G$. The \emph{crossed product quantum set} is defined as
    \begin{align*}
        (\tilde{B}, \tilde{\psi}) \rtimes_{\hat{\alpha}} \hat{G} := \left( \tilde{B} \rtimes_{\hat{\alpha}} \hat{G}, \psi_\rtimes \right),
    \end{align*}
    where $\tilde{B} \rtimes_{\hat{\alpha}} \hat{G}$ is the crossed product C*-algebra and $\psi_\rtimes$ is the functional given by
    \begin{align*}
        \psi_\rtimes\left( \sum_{\chi \in \hat{G}} b_\chi u_\chi \right) = n \tilde{\psi}(b_1)
    \end{align*}
    with $n := |G| =|\hat{G}|$.
\end{definition}

We will proceed to check that the above definition makes sense, i.e. we verify that $\psi_\rtimes$ is indeed a quantum set functional on the C*-algebra $\tilde{B} \rtimes_{\hat{\alpha}} \hat{G}$. Before we prove a useful lemma, let us introduce some notation.

\begin{notation}
    \label{pre::notation:crossed_product_quantum_sets}
    Given a quantum set $(\tilde{B}, \tilde{\psi})$ and an action $\hat{\alpha}: \hat{G} \to \mathrm{Aut}(\tilde{B}, \tilde{\psi})$ as above, we denote
    \begin{itemize}
        \item by $\tilde{m}$ the multiplication map on $\tilde{B}$ and by $\tilde{m}^\ast$ its adjoint with respect to the inner product induced by $\tilde{\psi}$,
        \item by $m_\rtimes$ the multiplication map on $\tilde{B} \rtimes_{\hat{\alpha}} \hat{G}$ and by $m_\rtimes^\ast$ its adjoint with respect to the inner product induced by $\psi_\rtimes$,
        \item by $E_\rtimes$ the linear map
        \begin{align*}
            E_\rtimes: \tilde{B} \rtimes_{\hat{\alpha}} \hat{G} \to \tilde{B}, \quad \sum_{\chi \in \hat{G}} b_\chi u_\chi \mapsto n b_1,
        \end{align*}
        and by $E_\rtimes^\ast$ its adjoint with respect to the inner products induced by the functionals $\tilde{\psi}$ and $\psi_\rtimes$ as above.
    \end{itemize}
\end{notation}

\begin{lemma}
    \label{land::lemma:description_of_crossed_product}
    Let $\hat{\alpha}$ be an action of $\hat{G}$ on the quantum set $(\tilde{B}, \tilde{\psi})$ as above. Then the following formulas are true for all $b, c \in \tilde{B}$ and $\chi, \xi \in \hat{G}$:
    \begin{align}
        \label{land::eq:formula_for_scalar_product_on_crossed_product_algebra}
        \left\langle b u_\chi, c u_\xi \right\rangle_{\psi_\rtimes} 
        &= n \delta_{\chi=\xi} \langle b, c \rangle_{\tilde{\psi}}, \\
        \label{land::eq:comultiplication_formula_on_crossed_product}
        m_\rtimes^\ast\left(b u_\chi\right) &= \frac{1}{n} \sum_{\xi \in \hat{G}} \sum_i b_i^{(1)} u_\xi \otimes \hat{\alpha}_{\xi^{-1}}(b_i^{(2)}) u_{\xi^{-1} \chi}, \\
        \label{eq:formula_for_E*}
        E_\rtimes^\ast(b) &= b,
    \end{align}
    where in \eqref{land::eq:comultiplication_formula_on_crossed_product} the elements $b_i^{(1)}, b_i^{(2)} \in \tilde{B}$ are such that $\tilde{m}^\ast(b) = \sum_i b_i^{(1)} \otimes b_i^{(2)}$.
\end{lemma}

\begin{proof}
    The first statement follows from
    \begin{align*}
        \left\langle b u_\chi, c u_\xi \right\rangle_{\psi_\rtimes} 
        &= \psi_\rtimes(u_\chi^\ast b^\ast c u_\xi)
        = n \delta_{\chi=\xi} \tilde{\psi}(\hat{\alpha}_{\chi^{-1}}(b^\ast c))
        = n \delta_{\chi=\xi} \langle b, c \rangle_{\tilde{\psi}},
    \end{align*}
    where we use $\tilde{\psi} = \tilde{\psi} \circ \hat{\alpha}_\cdot$ in the last step. Using \eqref{land::eq:formula_for_scalar_product_on_crossed_product_algebra}, we get for any elements $b u_\chi, c u_\rho, d u_\sigma \in \tilde{B} \rtimes_{\hat{\alpha}} \hat{G}$,
    \begin{align*}
        \frac{1}{n} \sum_\xi \sum_i& \langle b_i^{(1)} u_\xi \otimes \hat{\alpha}_{\xi^{-1}}(b_i^{(2)}) u_{\xi^{-1} \chi}, c u_\rho \otimes d u_\sigma \rangle_{\psi_\rtimes \otimes \psi_\rtimes} \\
        &= \frac{1}{n} \sum_\xi \sum_i \langle b_i^{(1)} u_\xi, c u_\rho \rangle_{\psi_\rtimes} \, \langle \hat{\alpha}_{\xi^{-1}}(b_i^{(2)}) u_{\xi^{-1} \chi}, d u_\sigma \rangle_{\psi_\rtimes} \\
        &= n \sum_\xi \sum_i \delta_{\xi=\rho} \langle b_i^{(1)}, c \rangle_{\tilde{\psi}} \, \delta_{\xi^{-1}\chi = \sigma} \langle \hat{\alpha}_{\xi^{-1}}(b_i^{(2)}), d \rangle_{\tilde{\psi}}.
    \end{align*}
    As $\tilde{\psi} = \tilde{\psi} \circ \hat{\alpha}_\cdot$ we have
    \begin{align*}
        \langle \hat{\alpha}_{\xi^{-1}}(b_i^{(2)}), d \rangle_{\tilde{\psi}} = \langle b_i^{(2)}, \hat{\alpha}_\xi(d) \rangle_{\tilde{\psi}}.
    \end{align*}
    Hence, we can continue the computation from above with
    \begin{align*}
        \dots 
        &= n \delta_{\chi=\rho \sigma} \sum_i \langle b_i^{(1)}, c \rangle_{\tilde{\psi}} \,\langle b_i^{(2)}, \hat{\alpha}_\rho(d) \rangle_{\tilde{\psi}} \\
        &= n \delta_{\chi=\rho \sigma} \langle m_\rtimes^\ast(b), c \otimes \hat{\alpha}_\rho(d) \rangle_{\tilde{\psi} \otimes \tilde{\psi}} \\
        &= n \delta_{\chi=\rho \sigma} \langle b, \tilde{m}(c \otimes \hat{\alpha}_\rho(d)) \rangle_{\tilde{\psi}} \\
        &= \langle b u_\chi, m_\rtimes(c u_\rho \otimes d u_\sigma) \rangle_{\psi_\rtimes}.
    \end{align*}
    Altogether this shows $m^\ast_\rtimes(b u_\chi) = \frac{1}{n} \sum_{\xi \in \hat{G}} \sum_i b_i^{(1)} u_\xi \otimes \hat{\alpha}_{\xi^{-1}}(b_i^{(2)}) u_{\xi^{-1} \chi}$.

    It remains to prove formula \eqref{eq:formula_for_E*} for the adjoint $E_\rtimes$. For this, let $b, c \in \tilde{B}$ and $\chi \in \hat{G}$ and use \eqref{land::eq:formula_for_scalar_product_on_crossed_product_algebra} to observe
    \begin{align*}
        \left\langle E_\rtimes\left(b u_\chi\right), c \right\rangle_{\tilde{\psi}} 
            = n \delta_{\chi=1} \langle b, c \rangle_{\tilde{\psi}}
            = \langle b u_\chi, c \rangle_{\psi_\rtimes}
            = \langle b u_\chi, E_\rtimes^\ast(c) \rangle_{\psi_\rtimes}.
    \end{align*}
\end{proof}

\begin{proposition}
    \label{crossed::prop:qset_functional_on_crossed_product}
    The map $\psi_\rtimes$ is a quantum set functional on the C*-algebra $\tilde{B} \rtimes_{\hat{\alpha}} \hat{G}$ if and only if $\tilde{\psi}$ is a quantum set functional on $\tilde{B}$.
\end{proposition}

\begin{proof}
    Using \eqref{land::eq:comultiplication_formula_on_crossed_product} from Lemma \ref{land::lemma:description_of_crossed_product} observe that for $b \in \tilde{B}$ and $\chi \in \hat{G}$,
    \begin{align*}
        m_\rtimes m_\rtimes^\ast\left(b u_\chi \right)
            &= \frac{1}{n} \sum_{\xi \in \hat{G}} \sum_i b_i^{(1)} u_\xi \hat{\alpha}_{\xi^{-1}}(b_i^{(2)}) u_{\xi^{-1} \chi} \\
            &= \sum_i b_i^{(1)} b_i^{(2)} u_\chi
    \end{align*}
    is equal to $b u_\chi$ if and only if $\sum_i b_i^{(1)} b_i^{(2)} = b$ holds. Thus, we have that $m_\rtimes m_\rtimes^\ast = \mathrm{Id}_{\tilde{B} \rtimes_{\hat{\alpha}} \hat{G}}$ is equivalent to $\tilde{m} \tilde{m}^\ast = \mathrm{Id}_{\tilde{B}}$. Evidently, $\psi_\rtimes$ is positive and faithful if and only if the same holds for $\tilde{\psi}$. Thus, the former is a quantum set functional if and only if the latter one is.
\end{proof}

If $\hat{\alpha}$ is the trivial action, then it is well known that the crossed product is simply a tensor product. Let us describe this situation in more detail to prepare for later calculations.

\begin{example}
    \label{crossed::example:tensor_product_quantum_set}
    Assume that $\hat{\alpha}: \hat{G} \to \mathrm{Aut}(\tilde{B}, \tilde{\psi})$ is the trivial action given by
    \begin{align*}
        \hat{\alpha}_\chi(b) = b \quad \text{ for all } b \in \tilde{B}, \chi \in \hat{G}.
    \end{align*}
    Then the crossed product $\tilde{B} \rtimes_{\hat{\alpha}} \hat{G}$ is actually isomorphic to the tensor product C*-algebra $\tilde{B} \otimes C(G)$, and an isomorphism is given by
    \begin{align*}
        \tilde{B} \rtimes_{\hat{\alpha}} \hat{G} &\to \tilde{B} \otimes C(G), \\
        \sum_{\chi \in \hat{G}} b_\chi u_\chi &\mapsto \sum_{\chi \in \hat{G}} b_\chi \otimes \tilde{u}_\chi,
    \end{align*}
    where $\tilde{u}_\chi = (\chi(g))_{g \in G} \in C(G)$ is the function on $G$ which evaluates to $\chi(g)$ at $g \in G$.
    Recall the functional $\psi_G$ on $C(G)$ from Example \ref{pre::ex:examples_of_quantum_sets} which is given by, equivalently,
    \begin{align*}
        \psi_G(e_g) = 1 \text{ for all } g \in G 
        \quad \text{ or } \quad 
        \psi_G\left( \tilde{u}_\chi \right) = n \delta_{\chi=1} \text{ for all } \chi \in \hat{G},
    \end{align*}
    where $n = |G|$. 
    The C*-algebra isomorphism from above identifies the functionals $\psi_\rtimes$ and $\tilde{\psi} \otimes \psi_G$.
    Indeed, we have for every $b \in \tilde{B}$ and $\chi \in \hat{G}$,
    \begin{align*}
        \psi_\rtimes(b u_\chi) = \delta_{\chi=1} n \tilde{\psi}(b),
    \end{align*}
    as well as
    \begin{align*}
        (\tilde{\psi} \otimes \psi_G)(b \otimes \tilde{u}_\chi) 
            = \delta_{\chi=1} n \tilde{\psi}(b).
    \end{align*}
    Summarizing, if $\hat{\alpha}$ is the trivial action, then we have an isomorphism of quantum sets
    \begin{align*}
        (\tilde{B} \otimes C(G), \tilde{\psi} \otimes \psi_G) \cong (\tilde{B} \rtimes_{\hat{\alpha}} \hat{G}, \psi_\rtimes).
    \end{align*}
    We denote the quantum set on the left-hand side by $(\tilde{B}, \tilde{\psi}) \otimes (C(G), \psi_G)$ and call it the \emph{tensor product quantum set} of $(\tilde{B}, \tilde{\psi})$ and $(C(G), \psi_G)$.
    Later we will prove the remarkable fact that for \emph{any} action $\hat{\alpha}$, which is not necessarily trivial, one has a \emph{quantum} isomorphism 
    $(\tilde{B}, \tilde{\psi}) \otimes (C(G), \psi_G) \cong (\tilde{B}, \tilde{\psi}) \rtimes_{\hat{\alpha}} \hat{G}$.
\end{example}

\subsection{Landstad Duality}

Let us finally turn to the question of how the operation $(\tilde{B}, \tilde{\psi}) \mapsto (\tilde{B}, \tilde{\psi}) \rtimes_{\hat{\alpha}} \hat{G}$ can be inverted. Thus, given a quantum set $(B, \psi)$ we want to find a smaller quantum set $(\tilde{B}, \tilde{\psi})$ together with an action $\hat{\alpha}: \hat{G} \to \mathrm{Aut}(\tilde{B}, \tilde{\psi})$ such that
\begin{align*}
    (B, \psi) \cong (\tilde{B}, \tilde{\psi}) \rtimes_{\hat{\alpha}} \hat{G}.
\end{align*}
For C*-algebras this question is answered by Landstad's Duality Theorem \cite{landstad_duality_1979}. This theorem works for general C*-algebras and von Neumann algebras, but we will only recall the particular version for finite-dimensional C*-algebras which is sufficient for our purposes.

\begin{thm}[{\cite[Theorem 2]{landstad_duality_1979}}]
    \label{pre::thm:landstads_thm}
    Let $B$ be a finite-dimensional C*-algebra and $\alpha: G \to \mathrm{Aut}(B)$ an action of a finite abelian group $G$ on $B$. Then $B$ is isomorphic to a crossed product C*-algebra over an action of $\hat{G}$ if and only if there exists a unitary representation $(u_\chi)_{\chi \in \hat{G}} \subset B$ such that
    \begin{align*}
        \alpha_g(u_\chi) = \chi(g)u_\chi
    \end{align*}
    holds for all $g \in G, \chi \in \hat{G}$. In this case, one has 
    \begin{align*}
        B \cong B^\alpha \rtimes_{\hat{\alpha}} \hat{G},
    \end{align*}
    where $B^\alpha = \{b \in B \mid \alpha_g(b) = b \text{ for all } g \in G\}$ is the fixed point algebra and $\hat{\alpha}: \hat{G} \to \mathrm{Aut}(B^\alpha)$ is the action defined by 
    \begin{align*}
        \hat{\alpha}_\chi(b) = u_\chi b u_\chi^\ast.
    \end{align*}
\end{thm}

It is not difficult to obtain a version of Landstad Duality for quantum sets by combining the above theorem with our discussion of crossed product quantum sets.

\begin{proposition}
    \label{pre::prop:landstad_for_quantum_sets}
    Let $(B, \psi)$ be a quantum set and $\alpha: G \to \mathrm{Aut}(B, \psi)$ be an action of a finite abelian group $G$ on $(B, \psi)$, which satisfies the conditions of Landstad's Duality Theorem (Theorem \ref{pre::thm:landstads_thm}). Additionally, assume that
    \begin{align}
        \label{land::eq:psi_invariant_under_adjoint_action}
        \psi = \psi \, \mathrm{Ad}(u_\chi)
    \end{align}
    holds for all $\chi \in \hat{G}$, where $(u_\chi)_{\chi \in \hat{G}} \subset B$ is the unitary representation from Landstad's theorem.
    Then one has an isomorphism of quantum sets
    \begin{align*}
        (B, \psi) \cong (B^\alpha, \tilde{\psi}) \rtimes_{\hat{\alpha}} \hat{G},
    \end{align*}
    where $\hat{\alpha}: \hat{G} \to \mathrm{Aut}(B^\alpha, \tilde{\psi})$ is the action from Landstad's theorem and $\tilde{\psi}$ is the quantum set functional on $B^\alpha$ given by
    \begin{align*}
        \tilde{\psi} := \frac{1}{n} \psi|_{B^\alpha}
    \end{align*}
    with $n := |G| = |\hat{G}|$.
\end{proposition}

\begin{proof}
    In view of Landstad's theorem, we only need to check that $\tilde{\psi}$ is a quantum set functional on $B^\alpha$ that is preserved by the action $\hat{\alpha}$, and that $\psi = \psi_\rtimes$ in the notation of the previous subsection. 

    First, let us check that $\tilde{\psi}$ is preserved by $\hat{\alpha}$ in the sense that
    \begin{align*}
        \tilde{\psi} = \tilde{\psi} \hat{\alpha}_\chi
    \end{align*}
    holds for all $\chi \in \hat{G}$. Due to \eqref{land::eq:psi_invariant_under_adjoint_action} we have for all $b \in B^\alpha$ and $\chi \in \hat{G}$
    \begin{align*}
        \tilde{\psi}(\hat{\alpha}_\chi(b)) 
            = \frac{1}{n} \psi(u_\chi b u_\chi^\ast) 
            = \frac{1}{n} \psi(b)
            = \tilde{\psi}(b).
    \end{align*}
    
    Next, we verify $\psi = \psi_\rtimes$ where $\psi_\rtimes$ is derived from $\tilde{\psi}$ as in Definition \ref{crossed::def:crossed_product_quantum_set}. Thus, we need to check that
    $$
        \psi(b u_\chi) = n \delta_{\chi=1} \tilde{\psi}(b)
    $$ 
    holds for all $b \in B^\alpha$ and $\chi \in \hat{G}$.
    Let us first prove this equation. Since $\alpha$ is an action on the quantum set $(B, \psi)$, we have for all $b \in B^\alpha$ and $\chi \in \hat{G}$
    \begin{align*}
        \psi(b u_\chi) = \psi(\alpha_g(b u_\chi)) = \chi(g) \psi(b u_\chi)
    \end{align*}
    for all $g \in G$. Hence, one has $\psi(b u_\chi) = 0$ whenever $\chi \neq 1$. It follows directly that
    \begin{align*}
        \psi(b u_\chi) = \delta_{\chi=1} \psi(b) = n \delta_{\chi=1} \tilde{\psi}(b)
    \end{align*}
    holds for all $b \in B^\alpha$ and $\chi \in \hat{G}$.
    
    In fact, this already proves that $\tilde{\psi}$ is a quantum set functional on $B^\alpha$: Proposition \ref{crossed::prop:qset_functional_on_crossed_product} guarantees that $\psi$ is a quantum set functional on $B$ if and only if $\tilde{\psi}$ is a quantum set functional on $B^\alpha$.
\end{proof}

\begin{remark}
    If $G$ is a finite group acting on a finite set $X$, then $X\cong X/G\times G$ as $G$-spaces, where the action on $X/G\times G$ is by the group multiplication on the right component. This fact plays a crucial role in the Gross--Tucker theorem for finite graphs. In general, if a compact group $G$ acts freely on a compact space $X$, the existence of the $G$-homeomorphism $X\cong X/G\times G$ is not automatic and is equivalent to the existence of a $G$-equivariant map $X\to G$~\cite{husemoller_bundles_1994}. In the noncommutative case, a $G$-equivariant map $X\to G$ translates to a unital $G$-equivariant $*$-homomorphism $C(G)\to A$, where $A$ is a unital C*-algebra, which is connected to the notion of a trivial quantum principal bundle~\cite{baum_hajac_dabrowski_bu_2015}. When $G$ is abelian, due to the Pontryagin duality, this condition is equivalent to the existence of a representation of the dual group $\hat{G}$ as in the assumptions of the Landstad theorem. Therefore, to obtain a quantum-graph version of the Gross--Tucker theorem, we have to rely on Theorem~\ref{pre::prop:landstad_for_quantum_sets}.
\end{remark}

\section{Voltage and derived quantum graphs}
\label{sec::qvoltage_and_derived_graphs}

In analogy with the classical theory of Gross and Tucker we introduce the notion of a \emph{voltage quantum graph} and its associated \emph{derived quantum graph}. The voltage quantum graph should be thought of as (not necessarily simple) quantum graph whose edges are labeled by elements of a finite abelian group $G$. Then, the derived quantum graph is a simple quantum graph on a larger crossed product quantum set that is constructed from the voltage quantum graph in a way that generalizes the classical derived graph construction.

Let us make the setting more precise. Throughout this section, $(\tilde{B}, \tilde{\psi})$ is a quantum set and $G$ is a finite abelian group. Further, let $\hat{\alpha}: \hat{G} \to \mathrm{Aut}(\tilde{B}, \tilde{\psi})$ be an action of the dual group $\hat{G}$ on $(\tilde{B}, \tilde{\psi})$, and let
\begin{align*}
    (B, \psi_\rtimes) := (\tilde{B}, \tilde{\psi}) \rtimes_{\hat{\alpha}} \hat{G}
    \quad \text{ with } \quad
    B = \tilde{B} \rtimes_{\hat{\alpha}} \hat{G}
\end{align*}
be the associated crossed product quantum set (see Definition \ref{crossed::def:crossed_product_quantum_set}). Recall the linear map $E_\rtimes: B \to \tilde{B}$ from Notation \ref{pre::notation:crossed_product_quantum_sets} as well as the unitary elements $(u_\chi)_{\chi \in \hat{G}} \subset B$ from Notation \ref{pre::notation:crossed_product}.

\begin{definition}
    \label{der::def:derived_quantum_graph}
    Let $(\tilde{B}, \tilde{\psi})$, $\hat{\alpha}: \hat{G} \to \mathrm{Aut}(\tilde{B}, \tilde{\psi})$, $(B, \psi_\rtimes)$ and $(u_\chi)_{\chi \in \hat{G}} \subset B$ be as above.
    \begin{enumerate}[label=(\alph*)]
        \item For every $g \in G$ let $X_g := \frac{1}{n} \sum_{\chi \in \hat{G}}\overline{\chi(g)} u_\chi u_\chi^\dagger$, i.e.
        \begin{align*}
            X_g : B \to B, \quad b \mapsto \frac{1}{n} \sum_{\chi \in \hat{G}} \overline{\chi(g)} \langle u_\chi, b \rangle_{\psi_\rtimes} u_\chi.
        \end{align*}
        \item A \emph{voltage quantum graph} on the quantum set $(\tilde{B}, \tilde{\psi})$ with respect to the group action $\hat{\alpha}: \hat{G} \to \mathrm{Aut}(\tilde{B}, \tilde{\psi})$ is a family $\tilde{A} = (\tilde{A}_g)_{g \in G}$ of quantum adjacency matrices that is indexed by $G$ and satisfies
        \begin{align}
            \label{der::eq:qvoltage_graph_invariance_under_group_action}
            \tilde{A}_g \hat{\alpha}_\chi = \hat{\alpha}_\chi \tilde{A}_g
        \end{align}
        for all $g \in G$ and $\chi \in \hat{G}$. 
        \item The \emph{derived quantum graph} $\tilde{A} \rtimes_{\hat{\alpha}} \hat{G}: B \to B$ on $(B, \psi_\rtimes)$ is then given by
        \begin{align*}
            \tilde{A} \rtimes_{\hat{\alpha}} \hat{G} := \sum_{g \in G} m(X_g \otimes E_\rtimes^\ast \tilde{A}_g E_\rtimes) m^\ast,
        \end{align*}
        where we write (here and in the rest of the section) $m$ for the multiplication map on $B$, i.e. $m = m_\rtimes$.
    \end{enumerate}
\end{definition}

The rest of this section is devoted to proving that this definition is well defined. Before that, let us discuss the case of the trivial action in more detail.

\begin{example}
    \label{der::ex:classical_voltage_graphs_as_quantum_voltage_graphs}
    Assume that $\hat{\alpha}: \hat{G} \to \mathrm{Aut}(\tilde{B}, \tilde{\psi})$ is the trivial action, i.e.
    \begin{align*}
        \hat{\alpha}_\chi = \mathrm{id}_{\tilde{B}} \quad \text{ for all } \chi \in \hat{G}.
    \end{align*}
    Then condition \eqref{der::eq:qvoltage_graph_invariance_under_group_action} is satisfied by any family of quantum adjacency matrices $(\tilde{A}_g)_{g \in G}$ on $(\tilde{B}, \tilde{\psi})$. Further, on the C*-algebra side the crossed product is isomorphic to the tensor product, i.e.
    \begin{align*}
        B = \tilde{B} \rtimes_{\hat{\alpha}} \hat{G} \cong \tilde{B} \otimes C(G).
    \end{align*}

    Recall from Example \ref{pre::ex:classical_graphs_as_quantum_graphs} that classical graphs on a finite set $V$ correspond to quantum graphs on the quantum set $(\C^V, \psi_V)$. If $(\tilde{B}, \tilde{\psi})$ is a quantum set of this form, then a voltage quantum graph on $(\tilde{B}, \tilde{\psi})$ with respect to the trivial action $\hat{\alpha}$ of a finite abelian dual group $\hat{G}$ can be identified with a classical voltage graph in the sense of Gross and Tucker. We will discuss this in more detail in the next section.
\end{example}

\begin{remark}
    \label{der::rem:intersection_of_qgraphs}
    The above definitions have a very visual interpretation which we will discuss later in more detail for the case where $(\tilde{B}, \tilde{\psi}) = (\C^V, \psi_V)$ is a classical set and $\hat{\alpha}$ is the trivial action.
    For now, let us only remark that the operation
    \begin{align}
        \label{der::eq:intersection_of_quantum_graphs}
        (X_g, E_\rtimes^\ast \tilde{A}_g E_\rtimes) \mapsto m(X_g \otimes E_\rtimes^\ast \tilde{A}_g E_\rtimes) m^\ast
    \end{align}
    can be interpreted as taking the intersection of the two quantum graphs $X_g$ and $E_\rtimes^\ast \tilde{A}_g E_\rtimes$. If $(\tilde{B}, \tilde{\psi})$ and $(B, \psi_\rtimes)$ are both classical sets and both $X_g$ and $E_\rtimes^\ast \tilde{A}_g E_\rtimes$ are classical graphs, then this operation always produces a new graph (or rather adjacency matrix) which has an edge between two vertices whenever both $X_g$ and $E_\rtimes^\ast \tilde{A}_g E_\rtimes$ contain this particular edge. However, in the truly quantum setting it is generally not guaranteed that two quantum adjacency matrices produces a new quantum adjacency matrix via \eqref{der::eq:intersection_of_quantum_graphs}. This is similar and related to the fact that, in a noncommutative C*-algebra, the product of two projections is not necessarily again a projection.
\end{remark}

We proceed to prove that the maps $X_g$ and $A$ from the previous definition are indeed quantum adjacency matrices on the quantum set $(B, \psi_\rtimes)$.

\begin{lemma}
    \label{der::lemma:commutation_in_derived_graph}
    We have for every $g \in G$ and $b u_\chi \in B$,
    \begin{align}
        \label{der::eq:formula_for_X_g}
        X_g(b u_\chi) &= \overline{\chi(g)} \tilde{\psi}(b) u_\chi, \\
        \label{der::eq:formula_for_m(X_g_otimes_EstarAE)mstar}
        m(X_g \otimes E_\rtimes^\ast \tilde{A}_g E_\rtimes) m^\ast(b u_\chi) &= \overline{\chi(g)} \tilde{A}_g(b) u_\chi, \\
        \label{der::eq:commutation_in_derived_graph}
        m(X_g \otimes E_\rtimes^\ast \tilde{A}_g E_\rtimes) m^\ast &= m(E_\rtimes^\ast \tilde{A}_g E_\rtimes \otimes X_g) m^\ast.
    \end{align}
\end{lemma}

\begin{proof}
    For the first equation, the definition of $X_g$ yields
    \begin{align*}
        X_g(b u_\chi) &= \frac{1}{n} \sum_{\xi \in \hat{G}} \overline{\xi(g)} u_\xi \, u_\xi^\dagger(b u_\chi),
    \end{align*}
    where $u_\xi^\dagger(b u_\chi) = \langle u_\xi, b u_\chi \rangle_{\psi_\rtimes} = n \delta_{\xi=\chi} \tilde{\psi}(b)$ by \eqref{land::eq:formula_for_scalar_product_on_crossed_product_algebra}. Hence, the sum over $\xi$ collapses, and we get
    \begin{align*}
        X_g(b u_\chi) &= \overline{\chi(g)} \tilde{\psi}(b) u_\chi.
    \end{align*}

    Next, let us use the formula for $m^\ast$ from \eqref{land::eq:comultiplication_formula_on_crossed_product} in \Cref{land::lemma:description_of_crossed_product} to compute the left-hand side of \eqref{der::eq:formula_for_m(X_g_otimes_EstarAE)mstar}. For any $b \in B$ and $\chi \in \hat{G}$ we get
    \begin{align*}
        m(X_g \otimes E_\rtimes^\ast \tilde{A}_g E_\rtimes) m^\ast(b u_\chi)
            &= \frac{1}{n} \sum_{\xi \in \hat{G}} X_g(b_i^{(1)} u_\xi) E_\rtimes^\ast \tilde{A}_g E_\rtimes(\hat{\alpha}_{\xi^{-1}}(b_i^{(2)}) u_{\xi^{-1} \chi}), \\
            &= \frac{1}{n} \sum_{\xi \in \hat{G}} \overline{\xi(g)} \tilde{\psi}(b_i^{(1)}) u_\xi \, E_\rtimes^\ast \tilde{A}_g E_\rtimes(\hat{\alpha}_{\xi^{-1}}(b_i^{(2)}) u_{\xi^{-1} \chi}).
    \end{align*}
    where $\tilde{m}^\ast(b) = \sum_i b_i^{(1)} \otimes b_i^{(2)}$ and we treat the sum over $i$ as implicit. For the second step we use \eqref{der::eq:formula_for_X_g} from above.
    The definition of $E_\rtimes$ and $E_\rtimes^\ast$ from Notation \ref{pre::notation:crossed_product_quantum_sets} and Lemma \ref{land::lemma:description_of_crossed_product} yield
    \begin{align*}
        E_\rtimes^\ast \tilde{A}_g E_\rtimes(\hat{\alpha}_{\xi^{-1}}(b_i^{(2)}) u_{\xi^{-1} \chi}) 
            &= n \delta_{\xi=\chi} \tilde{A}_g(\hat{\alpha}_{\chi^{-1}}(b_i^{(2)})).
    \end{align*}
    Thus, the sum over $\xi$ collapses, and we are left with
    \begin{align*}
        m(X_g \otimes E_\rtimes^\ast \tilde{A}_g E_\rtimes) m^\ast (b u_\chi) 
            &= \overline{\chi(g)} u_\chi \, \tilde{\psi}(b_i^{(1)}) \, \tilde{A}_g(\hat{\alpha}_{\chi^{-1}}(b_i^{(2)})).
    \end{align*}
    Since by assumption $\tilde{A}_g$ is invariant under the group action $\hat{\alpha}$, this expression is equal to 
    \begin{align*}
        \overline{\chi(g)} \tilde{\psi}(b_i^{(1)}) \tilde{A}_g(b_i^{(2)}) u_\chi
            = \overline{\chi(g)} \tilde{A}_g(b) u_\chi,
    \end{align*}
    for it is one of the first properties of the comultiplication that $\psi(b_i^{(1)}) T(b_i^{(2)}) = T(b)$ holds for any linear map $T: \tilde{B} \to \tilde{B}$, see Proposition \ref{pre::prop:properties_of_counit_and_comultiplication}.

    To complete the proof it suffices to show that $m(E_\rtimes^\ast \tilde{A}_g E_\rtimes \otimes X_g) m^\ast$ can be expressed via the same formula, i.e.
    \begin{align*}
        m(E_\rtimes^\ast \tilde{A}_g E_\rtimes \otimes X_g) m^\ast (b u_\chi) 
            &= \frac{1}{n} \overline{\chi(g)} \tilde{A}_g(b) u_\chi.
    \end{align*}
    This can be done with the exact same steps as above where one only needs to interchange the roles of $X_g$ and $E_\rtimes^\ast \tilde{A}_g E_\rtimes$.
\end{proof}

\begin{proposition}
    \label{der::prop:derived_graph_is_quantum_graph}
    The maps $X_g: B \to B$ and $A: B \to B$ from the previous definition are quantum adjacency matrices on the quantum set $(B, \psi_\rtimes)$.
\end{proposition}

\begin{proof}
    Let us first fix some $g \in G$ and prove that $X_g$ is a quantum adjacency matrix. For any $\chi, \xi \in \hat{G}$ one has
    \begin{align*}
        (u_\chi^\dagger \otimes u_\xi^\dagger) m^\ast = \left( m (u_\chi \otimes u_\xi) \right)^\dagger = u_{\chi \xi}^\dagger.
    \end{align*}
    Therefore, we obtain
    \begin{align*}
        m(X_g \otimes X_g) m^\ast
            &= \frac{1}{n^2} \sum_{\chi, \xi \in \hat{G}} \overline{\chi(g)} \, \overline{\xi(g)} \, m\left( (u_\chi \otimes u_\chi^\dagger) \otimes (u_\xi \otimes u_\xi^\dagger) \right) m^\ast \\
            &= \frac{1}{n^2} \sum_{\chi, \xi \in \hat{G}} \overline{\chi(g)} \, \overline{\xi(g)} \, u_{\chi \xi} \otimes u_{\chi \xi}^\dagger \\
            &= X_g,
    \end{align*}
    and this shows that $X_g$ is Schur-idempotent. What's more, it is an easy observation that a very similar calculation yields
    \begin{align*}
        m(X_g \otimes X_h) m^\ast = 0
    \end{align*}
    if $g \neq h$.
    It remains to show that $X_g$ is $\ast$-preserving. For any $b \in B$ and $\chi \in \hat{G}$ one first observes
    \begin{align*}
        u_\chi^\dagger(b^\ast) = \langle u_\chi, b^\ast \rangle_{\psi_\rtimes}
            = \overline{\langle u_\chi^\ast, b \rangle_{\psi_\rtimes}}
            = \overline{u_{\chi^{-1}}^\dagger (b)}.  
    \end{align*}
    This yields
    \begin{align*}
        X_g(b^\ast) &= \frac{1}{n} \sum_{\chi \in \hat{G}} \overline{\chi(g)} u_\chi \, u_\chi^\dagger(b^\ast) 
            = \frac{1}{n} \sum_{\chi \in \hat{G}} \chi^{-1}(g) u_{\chi^{-1}}^\ast \, \overline{u_{\chi^{-1}}^\dagger(b)} \\
            &= \left( \frac{1}{n} \sum_{\chi \in \hat{G}} \overline{\chi^{-1}(g)} u_{\chi^{-1}} \, u_{\chi^{-1}}^\dagger(b) \right)^\ast
            = X_g(b)^\ast,
    \end{align*}
    which is exactly the desired property. Thus, $X_g$ is a quantum adjacency matrix.

    Finally, let us prove that $A$ is a quantum adjacency matrix. Using the (co)associativity of $m$ (and $m^\ast$) as well as Lemma \ref{der::lemma:commutation_in_derived_graph}, we obtain
    \begin{align*}
        m(A \otimes A) m^\ast
            &= \frac{1}{n^2} \sum_{g, h \in G} m(m(X_g \otimes X_h)m^\ast \otimes m(E_\rtimes^\ast \tilde{A}_g E_\rtimes \otimes E_\rtimes^\ast \tilde{A}_h E_\rtimes)) m^\ast \\
            &= \frac{1}{n^2} \sum_{g\in G} m(X_g \otimes m(E_\rtimes^\ast \tilde{A}_g E_\rtimes \otimes E_\rtimes^\ast \tilde{A}_g E_\rtimes) m^\ast)m^\ast,
    \end{align*}
    since $m(X_g \otimes X_h) m^\ast = \delta_{g=h} X_g$. It is not hard to see $m(E_\rtimes^\ast \otimes E_\rtimes^\ast) = E_\rtimes^\ast m$. Further, one has
    \begin{align*}
        (E_\rtimes \otimes E_\rtimes) m^\ast = m^\ast E_\rtimes
    \end{align*}
    since for all $b u_\chi \in B$ it is
    \begin{align*}
        (E_\rtimes \otimes E_\rtimes) m^\ast(b u_\chi)
            &= \frac{1}{n} \sum_{\xi \in \hat{G}} E_\rtimes(b_i^{(1)} u_\xi) \otimes E_\rtimes(\hat{\alpha}_{\xi^{-1}}(b_i^{(2)}) u_{\xi^{-1} \chi}) \\
            &= n \delta_{\chi = 1} b_i^{(1)} \otimes b_i^{(2)}  \\
            &= n \delta_{\chi = 1} m^\ast(b) \\
            &= m^\ast E_\rtimes(b u_\chi).
    \end{align*}
    Thus, we have
    \begin{align*}
        m(E_\rtimes^\ast \tilde{A}_g E_\rtimes \otimes E_\rtimes^\ast \tilde{A}_g E_\rtimes) m^\ast
            &= E_\rtimes^\ast m(\tilde{A}_g \otimes \tilde{A}_g) m^\ast E_\rtimes
            = E_\rtimes^\ast \tilde{A}_g E_\rtimes,
    \end{align*}
    where we use that $\tilde{A}_g$ is a quantum adjacency matrix. Putting everything together, it follows that $A$ is Schur-idempotent. As both the $X_g$ and the maps $E_\rtimes^\ast \tilde{A}_g E_\rtimes$ are $\ast$-preserving, one obtains that $A$ is $\ast$-preserving as well.
\end{proof}

\subsection{Properties of the derived quantum graph}

In this section we study three properties of the derived quantum graph $\tilde{A} \rtimes_{\hat{\alpha}} \hat{G}$ relative to properties of the voltage quantum graph $\tilde{A}$. These properties are: being loopfree, being undirected and being $d$-regular for some $d \geq 0$. The definitions of these properties are well-known in the literature see e.g. \cite{musto_compositional_2018,matsuda_classification_2022}. If $\psi$ is a non-tracial functional there is an ambiguity in the definition of undirectedness but we will simply stick to the same definition as for the tracial case. An alternative notion can be found in \cite{wasilewski_quantum_2024}, see also \cite{daws_quantum_2024}.

\begin{definition}
    Let $A$ be a quantum graph on a quantum set $(B, \psi)$. We call the quantum graph
    \begin{enumerate}
        \item \emph{loopfree} if $m(A \otimes \mathrm{Id})m^\ast = 0$,
        \item \emph{undirected} if $A = A^\ast$ where the adjoint is taken with respect to the inner product induced by $\psi$, and
        \item \emph{$d$-regular} for some $d \geq 0$ if $A(1) = d$.
    \end{enumerate}
    We say that a voltage quantum graph $\tilde{A} = (\tilde{A}_g)_{g \in G}$ has these properties if the same equations hold for $\sum_{g \in G} \tilde{A}_g$, respectively.
\end{definition}

Though $\sum_{g \in G} \tilde{A}_g$ is not a quantum adjacency matrix it makes sense to think of this matrix as the adjacency matrix of a quantum multigraph. We don't make this notion precise here but refer to the discussion in \cite[Section 2]{gromada_examples_2022}. For more information on regular quantum graphs we refer to \cite{matsuda_algebraic_2024}.

The following proposition relates properties of a derived quantum graph $\tilde{A} \rtimes_{\hat{\alpha}} \hat{G}$ to properties of the underlying voltage quantum graph $\tilde{A}$. 

\begin{proposition}\label{prop:voltage_to_derived_properties}
    Let $\tilde{A} = (\tilde{A}_g)_{g \in G}$ be a voltage quantum graph on a quantum set $(\tilde{B}, \tilde{\psi})$ with respect to an action $\hat{\alpha}$ of a finite abelian group $\hat{G}$. Then the derived quantum graph $\tilde{A} \rtimes_{\hat{\alpha}} \hat{G}$ is
    \begin{enumerate}
        \item loopfree if $\tilde{A}_1$ is loopfree,
        \item undirected if $\tilde{A}_g^\ast = \tilde{A}_{g^{-1}}$ holds for all $g \in G$, and
        \item $d$-regular if $\tilde{A}$ is $d$-regular, i.e. if $\sum_{g \in G} \tilde{A}_g(1) = d$.
    \end{enumerate}
\end{proposition}

\begin{proof}
    Let $A := \tilde{A} \rtimes_{\hat{\alpha}} \hat{G}$.

    Ad (a). Using the definition of $A$ together with co-associativity of $m$ and \eqref{der::eq:commutation_in_derived_graph} we get
    \begin{align*}
        m(A \otimes \mathrm{Id}) m^\ast
            &= \sum_{g \in G} m(E_\rtimes^\ast \tilde{A}_g E_\rtimes \otimes m(X_g \otimes \mathrm{Id})m^\ast) m^\ast.
    \end{align*}
    A simple calculation using the comultiplication formula \eqref{land::eq:comultiplication_formula_on_crossed_product} reveals for $b \in \tilde{B}$ and $\chi \in \hat{G}$ that
    \begin{align*}
        m(X_g \otimes \mathrm{Id}) m^\ast(b u_\chi)
            &= \frac{1}{n} \sum_{\zeta \in \hat{G}} X_g(b_i^{(1)} u_\zeta) \hat{\alpha}_{\zeta^{-1}}(b_i^{(2)}) u_{\zeta^{-1} \chi} \\
            &= \frac{1}{n} \sum_{\zeta \in \hat{G}} \overline{\zeta(g)} \tilde{\psi}(b_i^{(1)}) b_i^{(2)} u_{\chi} 
            = \delta_{g=1} b u_\chi.
    \end{align*}
    Thus, the above expression becomes
    \begin{align*}
        m(A \otimes \mathrm{Id}) m^\ast
            &= m(E_\rtimes^\ast \tilde{A}_1 E_\rtimes \otimes \mathrm{Id}) m^\ast.
    \end{align*}
    Using again formula \eqref{land::eq:comultiplication_formula_on_crossed_product} one gets in a similar way
    \begin{align*}
        m(E_\rtimes^\ast \tilde{A}_1 E_\rtimes \otimes \mathrm{Id}) m^\ast(b u_\chi)
            &= \tilde{A}_1(b_i^{(1)}) b_i^{(2)} u_\chi \\
            &= \tilde{m}(\tilde{A}_1 \otimes \mathrm{Id}) \tilde{m}^\ast(b) u_\chi.
    \end{align*}
    The statement follows immediately.

    Ad (b). First it is an easy observation that $X_g^\ast = X_{g^{-1}}$ since
    \begin{align*}
        X_g^\ast = \frac{1}{n} \sum_{\chi \in \hat{G}} (\overline{\chi(g)} u_\chi u_\chi^\dagger)^\ast
            = \frac{1}{n} \sum_{\chi \in \hat{G}} \chi(g) u_\chi u_\chi^\dagger
            = X_{g^{-1}}.
    \end{align*}
    Then we have
    \begin{align*}
        A^\ast = \sum_{g \in G} m(X_g^\ast \otimes E_\rtimes^\ast \tilde{A}_g^\ast E_\rtimes) m^\ast,
    \end{align*}
    which shows that $\tilde{A}_g^\ast = \tilde{A}_{g^{-1}}$ for all $g \in G$ entails that $A^\ast = A$.

    Ad (c). Using \eqref{der::eq:formula_for_m(X_g_otimes_EstarAE)mstar} from Lemma \ref{der::lemma:commutation_in_derived_graph} it is a simple observation that
    \begin{align*}
        A(1) = \sum_{g \in G} \tilde{A}_g(1) = d
    \end{align*}
    holds if $\tilde{A}$ is $d$-regular.
\end{proof}

\subsection{Connection to Gross and Tucker's construction}

Let us now compare our definition of voltage quantum graphs and their derived quantum graphs with Gross and Tucker's classical theory. By Example \ref{pre::ex:classical_graphs_as_quantum_graphs} a classical (simple, directed) graph with vertex set $V$ corresponds to a quantum graph on the quantum set $(\C^V, \psi_V)$. We will see that in the same vein, voltage quantum graphs on $(\C^V, \psi_V)$ with respect to the trivial action of $\hat{G}$ correspond to classical voltage graphs, and the derived quantum graph corresponds to Gross and Tucker's derived graph.

A classical directed (not necessarily simple) graph $\Gamma$ is given by a (finite) vertex set $V$ and a (finite) edge set $E$ together with maps $s, t: E \to V$ which assign to every edge its starting vertex and terminal vertex, respectively. If $s(e) = t(e)$, then the edge $e$ is a loop. If $s(e) = v \in V$ and $t(e) = u \in V$, we suggestively write $e: v {\to} u$.
Let us recall the following definitions from Gross and Tucker.

\begin{definition}[{\cite[Section 2.1.1]{gross_topological_1987}}]
    A \emph{voltage graph} is a pair $(\Gamma, \lambda)$ consisting of a (not necessarily simple) finite directed graph $\Gamma$ with vertex set $V$ and edge set $E$, together with a labeling $\lambda: E \to G$ of the edges of $\Gamma$ by elements of a finite group $G$.

    The \emph{(right) derived graph} $\Gamma^\lambda$ of $(\Gamma, \lambda)$ has vertex set $V \times G$ and edge set $E \times G$. Whenever $e \in E$ is an edge from $u$ to $v$ in $\Gamma$ and $g \in G$ is a group element, then the edge $(e,g)$ goes from $(u,g)$ to $(v, g \lambda(e))$ in $\Gamma^\lambda$, i.e.
    \begin{align*}
        e: u \to v \text{ in } \Gamma \text{ and } g \in G
        \quad \implies \quad
        (e,g): (u,g) \to (v, g \lambda(e)) \text{ in } \Gamma^\lambda.
    \end{align*}
\end{definition}

In the above definition neither the voltage graph nor the derived graph need to be simple.
However, our definition of a voltage quantum graph has been designed so that the derived quantum graph is always simple in the sense that the derived quantum graph is given by a single quantum adjacency matrix. This is intended to reduce technical complications and because one is mostly interested in simple quantum graphs. However, Gross and Tucker's construction does not make similar provisions. Therefore, we can only hope to recover their construction in the case where the derived graph is simple. The next proposition characterizes when this is the case.

\begin{lemma}
    \label{der::lemma:presimple_voltage_graph}
    The derived graph $\Gamma^\lambda$ of a voltage graph $(\Gamma, \lambda)$ is simple if and only if for any two vertices $u, v \in V$ and group element $g \in G$ there is at most one edge $e: u \to v$ that is labeled by $g$.
\end{lemma}

\begin{proof}
    Assume that $\Gamma^\lambda$ is not simple, and let $(e,g), (f,h)$ be two edges that have the same starting vertex and the same terminal vertex in $\Gamma^\lambda$. It is not difficult to check that this implies that $g=h$ and that $e$ and $f$ have the same starting vertex and the same terminal vertex in $\Gamma$ as well. Further, they must be labeled by the same group element. Thus, 
    the statement follows by contraposition.
\end{proof}

We call a voltage graph \emph{pre-simple} if it satisfies the assumption of Proposition~\ref{der::lemma:presimple_voltage_graph}.
 
\begin{proposition}
    \label{der::prop:classical_voltage_graphs_as_quantum_voltage_graphs}
    Let $G$ be a finite abelian group and $(\C^V, \psi_V)$ the quantum set from Example \ref{pre::ex:classical_graphs_as_quantum_graphs} for some finite set $V$.

    Given a pre-simple voltage graph $(\Gamma, \lambda)$ on $V$ with labeling $\lambda:E \to G$, let $\tilde{A}_g: \C^V \to \C^V$ be the adjacency matrix of the (simple, directed) subgraph of $\Gamma$ that contains all edges from $\Gamma$ that are labeled by $g$. Then the map
    \begin{align*}
        (\Gamma, \lambda) \mapsto \tilde{A} = (\tilde{A}_{g})_{g \in G}
    \end{align*}
    is a one-to-one correspondence between $G$-labeled pre-simple voltage graphs on $V$ and voltage quantum graphs on $(\C^V, \psi_V)$ with respect to the trivial action on $(\C^V, \psi_V)$.
\end{proposition}

\begin{proof}
    As observed in Example \ref{der::ex:classical_voltage_graphs_as_quantum_voltage_graphs} a voltage quantum graph with respect to the trivial action is nothing else than a family $(\tilde{A}_g)_{g \in G}$ of quantum adjacency matrices on some quantum set. Quantum adjacency matrices on $(\C^V, \psi_V)$ are the same as adjacency matrices of (simple, directed) graphs on $V$ by Example \ref{pre::ex:classical_graphs_as_quantum_graphs}. Thus, it is evident that every pre-simple classical voltage graph can be interpreted as a voltage quantum graph on $(\C^V, \psi_V)$ as in the statement, and vice versa.
\end{proof}

\begin{lemma}
    \label{der::lemma:identification_of_derived_graph_components_classical_case}
    Let $(\C^V, \psi_V)$ be the quantum set from Example \ref{pre::ex:classical_graphs_as_quantum_graphs}. Further, let $G$ be a finite abelian group with the trivial action $\mathrm{triv}: \hat{G} \to \mathrm{Aut}(\C^V, \psi_V)$. The following are true:
    \begin{enumerate}
        \item The assignment from the generators of $\C^V \rtimes_{\mathrm{triv}} \hat{G}$ into $C(V \times G)$ given by
        \begin{align*}
            \begin{aligned}
                e_v \mapsto \sum_{g \in G} e_{v,g}, \qquad
                u_\chi \mapsto \sum_{v \in V, g \in G} \chi(g) e_{v,g}
            \end{aligned}
        \end{align*}
        extends to an isomorphism
        \begin{align*}
            (\C^V, \psi_V) \rtimes_{\mathrm{triv}} \hat{G} \cong (\C^{V \times G}, \psi_{V \times G}),
        \end{align*}
        where the quantum set on the right-hand side is associated to the classical finite set $V \times G$ as in Example \ref{pre::ex:classical_graphs_as_quantum_graphs}, and we write $e_v$, $e_{v, g}$ for indicator functions in $C(V)$ and $C(V \times G)$, respectively.
        \item The quantum graph $X_g: B \to B$ with $g \in G$ from Definition \ref{der::def:derived_quantum_graph} corresponds to the classical graph on $V \times G$ that contains all edges of the form
        \begin{align*}
            (v, h) \to (w, h g)
            \qquad
            \text{ for } v, w \in V \text{ and } h \in G.
        \end{align*}
        \item The quantum graph $E_\rtimes^\ast \tilde{A}_g E_\rtimes: B \to B$ from Definition \ref{der::def:derived_quantum_graph} corresponds to the classical graph on $V \times G$ that contains all edges of the form
        \begin{align*}
            (v, h) \to (w, k)
            \qquad
            \text{ for } v \overset{\tilde{A}_g}{\longrightarrow} w \in V \text{ and } h, k \in G,
        \end{align*}
        where $v \overset{\tilde{A}_g}{\longrightarrow} w$ indicates that there is an edge from $v$ to $w$ in the classical graph corresponding to the adjacency matrix $\tilde{A}_g$.
        \item Let $A_1, A_2$ be two quantum adjacency matrices which correspond to classical graphs $\Gamma_1, \Gamma_2$ on $V \times G$ in the sense of Example \ref{pre::ex:classical_graphs_as_quantum_graphs}. Then, the map
        \begin{align*}
            A_3 := m(A_1 \otimes A_2) m^\ast
        \end{align*}
        is again a quantum adjacency matrix, and it corresponds to the classical graph
        \begin{align*}
            \Gamma_3 := \Gamma_1 \cap \Gamma_2,
        \end{align*}
        which is the intersection of $\Gamma_1$ and $\Gamma_2$, i.e. $\Gamma_3$ has vertex set $V \times G$ and its edge set is the intersection of the edge sets of $\Gamma_1$ and $\Gamma_2$.
    \end{enumerate}
\end{lemma}

\begin{proof}
    For the first statement, it suffices to recall from Example \ref{crossed::example:tensor_product_quantum_set} that
    \begin{align*}
        (\C^V, \psi_V) \rtimes_{\mathrm{triv}} \hat{G} \cong (\C^V \otimes C(G), \psi_V \otimes \psi_G),
    \end{align*}
    where an isomorphism is given by
    \begin{align*}
        e_v u_\chi \mapsto \sum_{g \in G} \chi(g) e_{v,g}.
    \end{align*}

    For (b), let us identify $u_\chi$ with $\sum_{g \in G} \chi(g) e_{v,g}$ according to the above isomorphism. Then, one has for all $\chi \in \hat{G}, v \in V$ and $h \in G$
    \begin{align*}
        u_\chi^\dagger(e_{v,h})
            = \langle u_\chi, e_{v,h} \rangle_{\psi_{V \otimes G}}
            = \left\langle \sum_{w \in V, k \in G} \chi(k) e_{w,k}, e_{v,h} \right\rangle_{\psi_{V \otimes G}}
            = \overline{\chi(h)}.
    \end{align*}
    It follows
    \begin{align*}
        X_g(e_{v, h}) &= \frac{1}{n} \sum_{\chi \in \hat{G}} \overline{\chi(g)} u_\chi \, u_\chi^\dagger(e_{v, h})
            = \frac{1}{n} \sum_{\chi \in \hat{G}} \sum_{w \in V, k \in G} \overline{\chi(g h)} \chi(k) e_{w, k}
            = \sum_{w \in V} e_{w, h g},
    \end{align*}
    where we use in the last step that $\sum_{\chi \in \hat{G}} \chi(h^{-1} g^{-1} k) = n$ if $k = h g$ and $0$ otherwise. The claim follows directly.

    Next, let us prove (3). A short calculation suffices to verify
    \begin{align*}
        E_\rtimes(e_v \otimes e_h) = e_v,
        \quad \text{ and } \quad
        E_\rtimes^\ast(e_v) = e_v \otimes 1_{C(G)} = \sum_{k \in G} e_v \otimes e_k,
    \end{align*}
    where $E_\rtimes$ and $E_\rtimes^\ast$ are as in Notation \ref{pre::notation:crossed_product_quantum_sets}.
    It follows
    \begin{align*}
        E_\rtimes^\ast \tilde{A}_g E_\rtimes(e_v \otimes e_h) 
            = \tilde{A}_g(e_v) \otimes 1_{C(G)}
            = \sum_{w \sim_{\tilde{A}_g} v} \sum_{k \in G} e_w \otimes e_k,
    \end{align*}
    and this concludes the proof.

    Finally, (4) is an easy computation. Indeed, let $A_1, A_2$ be quantum adjacency matrices on the quantum set $(\C^{V}, \psi_{V})$ that correspond to classical graphs $\Gamma_1, \Gamma_2$ on $V$. Then, for any $v \in V$ one has
    \begin{align*}
        A_3(e_{v}) = m(A_1 \otimes A_2) m^\ast (e_{v})
            = A_1(e_v) A_2(e_v)
            = \sum_{w \sim_{\Gamma_1 \cap \Gamma_2} v} e_w,
    \end{align*}
    since the comultiplication on $\C^V$ is given $m^\ast(e_v) = e_v \otimes e_v$ by Example \ref{pre::ex:examples_of_quantum_sets}.
\end{proof}

Now, we are ready to see that our derived quantum graph construction generalizes the classical construction from pre-simple voltage graphs.

\begin{proposition}
    In the situation of Proposition \ref{der::prop:classical_voltage_graphs_as_quantum_voltage_graphs} one has for a pre-simple voltage graph $(\Gamma, \lambda)$ and the associated voltage quantum graph $\tilde{A} = (\tilde{A}_g)_{g \in G}$ the identification
    \begin{align*}
        \Gamma^\lambda = \tilde{A} \rtimes_{\mathrm{triv}} \hat{G}
    \end{align*}
    of the derived graphs, where $\Gamma^\lambda$ is Gross and Tucker's derived graph, $\tilde{A} \rtimes_{\mathrm{triv}} \hat{G}$ is the derived quantum graph from Definition \ref{der::def:derived_quantum_graph}, and we identify classical graphs on $V \times G$ with quantum graphs on $(\C^V, \psi_V) \rtimes_{\mathrm{triv}} \hat{G}$ according to Example \ref{pre::ex:classical_graphs_as_quantum_graphs}.
\end{proposition}

\begin{proof}
    Evidently, the derived quantum graph $\tilde{A} \rtimes_{\mathrm{triv}} \hat{G}$ is defined on the quantum set 
    $
        (\C^V, \psi_V) \rtimes_{\mathrm{triv}} \hat{G},
    $
    and by Lemma \ref{der::lemma:identification_of_derived_graph_components_classical_case} (a) this quantum set is isomorphic to $(C(V \times G), \psi_{V \times G})$. Thus, the derived quantum graph corresponds to some classical graph on the set $V \times G$ (due to Example \ref{pre::ex:classical_graphs_as_quantum_graphs}). It remains to check that this classical graph is exactly Gross and Tucker's derived graph $\Gamma^\lambda$.

    Recall from Definition \ref{der::def:derived_quantum_graph} that
    \begin{align*}
        \tilde{A} \rtimes_{\mathrm{triv}} \hat{G} = \sum_{g \in G} m\left( X_g \otimes E_\rtimes^\ast \tilde{A}_g E_\rtimes \right) m^\ast.
    \end{align*}
    In view of Lemma \ref{der::lemma:identification_of_derived_graph_components_classical_case}, for any fixed $g \in G$ the map
    \begin{align*}
        m\left( X_g \otimes E_\rtimes^\ast \tilde{A}_g E_\rtimes \right) m^\ast
    \end{align*}
    corresponds to the classical graph on $V \times G$ that contains all edges of the form
    \begin{align*}
        (v, h) \to (w, hg)
        \qquad
        \text{ for } v \overset{\tilde{A}_g}{\longrightarrow} w \in V \text{ and } h \in G.
    \end{align*}
    It follows immediately that the derived quantum graph $\tilde{A} \rtimes_{\mathrm{triv}} \hat{G}$ corresponds to the classical graph on $V \times G$ that contains all edges of the form
    \begin{align*}
        (v, h) \to (w, h \lambda(e))
        \qquad
        \text{ for } e: v \to w \in \Gamma \text{ and } h \in G,
    \end{align*}
    and this is exactly Gross and Tucker's derived graph $\Gamma^\lambda$.
\end{proof}

Finally, let us illustrate the previous proposition on a concrete example.

\begin{example}
    Consider the graph $\Gamma$ with one vertex $v$ and one loop $e$ based at $v$. Let us label the loop with the generator $1_{\Z_3}$ of the group $\Z_3 = \{0_{\Z_3}, 1_{\Z_3}, 2_{\Z_3}\}$. Thus, we have a voltage graph $(\Gamma, \lambda)$ with $\lambda(e) = 1_{\Z_3}$. Then, the derived graph $\Gamma^\lambda$ has three vertices $v_0, v_1, v_2$ corresponding to the elements of $\Z_3$ and three edges as sketched in Figure \ref{der::fig:voltage_graph_example}.
    \begin{figure}[bt]
        \begin{tikzpicture}[
            node distance=1.5cm,
            vertex/.style={circle, fill, inner sep=1pt},
            >=Stealth
          ]
          \begin{scope}
            \node at (0, -1) {$\Gamma$};
            \node[vertex, label=below:$v$] (g1) at (0,0) {};
            \path[scale=1.5, ->] (g1) edge [loop above] node {$1_{\mathbb{Z}_3}$} (g1);
          \end{scope}

          \begin{scope}[xshift=4cm]
            \node at (0, -1) {$\Gamma^\lambda$};
            \node[vertex, label=below:$v_1$] (b) at (0,0) {};
            \node[vertex, label=below:$v_0$] (a) [left=of b] {};
            \node[vertex, label=below:$v_2$] (c) [right=of b] {};
        
            \draw[->] (a) to (b);
            \draw[->] (b) to (c);
            \draw[->] (c) to [bend right=45] (a);
          \end{scope}
        \end{tikzpicture}
        \caption{The voltage graph $(\Gamma, \lambda)$ and its derived graph $\Gamma^\lambda$.}
        \label{der::fig:voltage_graph_example}
    \end{figure}
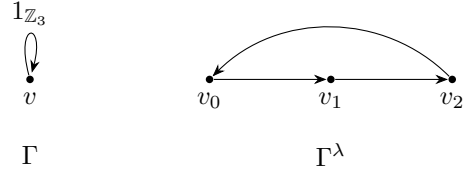

    As discussed above, a classical graph on $n$ vertices corresponds to a quantum graph on a quantum set $(\C^n, \psi_n)$.
    The voltage graph $(\Gamma, \lambda)$ is identified with a voltage quantum graph on $(\C, \psi_1)$ which is given by a family $(\tilde{A}_g)_{g \in \Z_3}$ of quantum adjacency matrices on $(\C, \psi_1)$. Since $\Gamma$ has only one edge which is labeled by $1_{\Z_3}$, it is not hard to check that $(\Gamma, \lambda)$ corresponds to the family $\tilde{A} = (\tilde{A}_g)_{g \in \Z_3}$ given by
    \begin{align*}
        \tilde{A}_{0_{\Z_3}} = 0, \quad \tilde{A}_{1_{\Z_3}} = \mathrm{id}_{\C}, \quad \tilde{A}_{2_{\Z_3}} = 0.
    \end{align*}
    Generally, we would have to check that $(\tilde{A}_g)_{g \in \Z_3}$ satisfies the conditions from Definition \ref{der::def:derived_quantum_graph} to be a voltage quantum graph. However, in this case we consider $(\tilde{A}_g)_{g \in \Z_3}$ as a voltage quantum graph with respect to the trivial action $\mathrm{triv}$ of $\hat{\Z}_3$ on $\C$; thus, there is nothing to check (see Example \ref{der::ex:classical_voltage_graphs_as_quantum_voltage_graphs}).

    What is the derived quantum graph $\tilde{A} \rtimes_{\mathrm{triv}} \hat{\Z}_3$ of this voltage quantum graph? First, recall that the derived quantum graph is defined on the crossed product quantum set
    \begin{align*}
        (\C, \psi_1) \rtimes_{\mathrm{triv}} \hat{\Z}_3.
    \end{align*}
    By Lemma \ref{der::lemma:identification_of_derived_graph_components_classical_case} (a) this quantum set is isomorphic to $(\C^3, \psi_3)$, and quantum graphs on $(\C^3, \psi_3)$ correspond to classical graph on three vertices.

    Next, by Definition \ref{der::def:derived_quantum_graph} the derived quantum graph $A := \tilde{A} \rtimes_{\mathrm{triv}} \hat{\Z}_3: \C^3 \to \C^3$ on $(\C^3, \psi_3)$ is given by the formula
    \begin{align*}
        A = \sum_{g \in \Z_3} m(X_g \otimes E_\rtimes^\ast \tilde{A}_g E_\rtimes) m^\ast.
    \end{align*}
    According to Lemma \ref{der::lemma:identification_of_derived_graph_components_classical_case} (b) and (c) the quantum graphs $X_g$ and $E_\rtimes^\ast \tilde{A}_g E_\rtimes$ correspond to the classical graphs on the set $\{v_0, v_1, v_2\}$ in Figure \ref{der::fig:Xg_graphs_example} and Figure \ref{der::fig:EAE_graph_example}, respectively.
    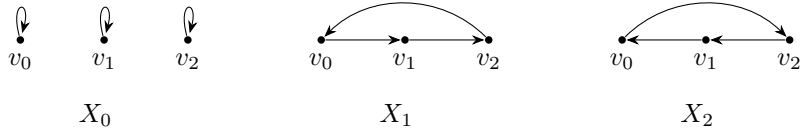
\begin{figure}[bt]
        \begin{tikzpicture}[
            node distance=1cm,
            vertex/.style={circle, fill, inner sep=1pt},
            >=Stealth
          ]
          \begin{scope}
            \node at (1, -1) {$X_0$};
            \node[vertex, label=below:$v_0$] (a0) at (0,0) {};
            \node[vertex, label=below:$v_1$] (b0) [right=of a0] {};
            \node[vertex, label=below:$v_2$] (c0) [right=of b0] {};
        
            \path[->] (a0) edge [loop above] (a0);
            \path[->] (b0) edge [loop above] (b0);
            \path[->] (c0) edge [loop above] (c0);
          \end{scope}

          \begin{scope}[xshift=4cm]
            \node at (1, -1) {$X_1$};
            \node[vertex, label=below:$v_0$] (a1) at (0,0) {};
            \node[vertex, label=below:$v_1$] (b1) [right=of a1] {};
            \node[vertex, label=below:$v_2$] (c1) [right=of b1] {};
        
            \draw[->] (a1) to (b1);
            \draw[->] (b1) to (c1);
            \draw[->] (c1) to [bend right=45] (a1);
          \end{scope}

          \begin{scope}[xshift=8cm]
            \node at (1, -1) {$X_2$};
            \node[vertex, label=below:$v_0$] (a2) at (0,0) {};
            \node[vertex, label=below:$v_1$] (b2) [right=of a2] {};
            \node[vertex, label=below:$v_2$] (c2) [right=of b2] {};
        
            \draw[->] (a2) to [bend left=45] (c2);
            \draw[->] (b2) to (a2);
            \draw[->] (c2) to (b2);
          \end{scope}
        \end{tikzpicture}
        \caption{The graphs corresponding to the adjacency matrices $X_g$.}
        \label{der::fig:Xg_graphs_example}
    \end{figure}
    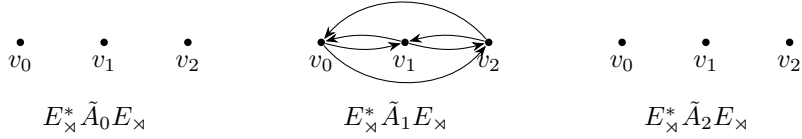
\begin{figure}[bt]
        \begin{tikzpicture}[
            node distance=1cm,
            vertex/.style={circle, fill, inner sep=1pt},
            >=Stealth
          ]
          \begin{scope}
            \node at (1, -1) {$E_\rtimes^\ast \tilde{A}_0 E_\rtimes$};
            \node[vertex, label=below:$v_0$] (a0) at (0,0) {};
            \node[vertex, label=below:$v_1$] (b0) [right=of a0] {};
            \node[vertex, label=below:$v_2$] (c0) [right=of b0] {};
          \end{scope}

          \begin{scope}[xshift=4cm]
            \node at (1, -1) {$E_\rtimes^\ast \tilde{A}_1 E_\rtimes$};
            \node[vertex, label=below:$v_0$] (a1) at (0,0) {};
            \node[vertex, label=below:$v_1$] (b1) [right=of a1] {};
            \node[vertex, label=below:$v_2$] (c1) [right=of b1] {};
        
            \draw[->] (a1) to [bend right=15] (b1);
            \draw[->] (b1) to [bend right=15] (a1);
            \draw[->] (a1) to [bend right=50] (c1);
            \draw[->] (c1) to [bend right=50] (a1);
            \draw[->] (b1) to [bend right=15] (c1);
            \draw[->] (c1) to [bend right=15] (b1);
          \end{scope}

          \begin{scope}[xshift=8cm]
            \node at (1, -1) {$E_\rtimes^\ast \tilde{A}_2 E_\rtimes$};
            \node[vertex, label=below:$v_0$] (a2) at (0,0) {};
            \node[vertex, label=below:$v_1$] (b2) [right=of a2] {};
            \node[vertex, label=below:$v_2$] (c2) [right=of b2] {};
          \end{scope}
        \end{tikzpicture}
        \caption{The graphs corresponding to the adjacency matrices $E_\rtimes^\ast \tilde{A}_g E_\rtimes$.}
        \label{der::fig:EAE_graph_example}
    \end{figure}
    By Lemma \ref{der::lemma:identification_of_derived_graph_components_classical_case} (d) the term
    \begin{align*}
        m (X_g \otimes E_\rtimes^\ast \tilde{A}_g E_\rtimes) m^\ast
    \end{align*}
    returns the adjacency matrix of the classical graph that is obtained by taking the intersection of the graphs from Figure \ref{der::fig:Xg_graphs_example} and Figure \ref{der::fig:EAE_graph_example}. Thus, only the term for $g = 1_{\Z_3}$ contributes to the derived quantum graph $A$, and we have
    \begin{align*}
        A = m (X_1 \otimes E_\rtimes^\ast \tilde{A}_1 E_\rtimes) m^\ast.
    \end{align*}
    It is not hard to see that this graph is exactly the derived graph $\Gamma^\lambda$ from Figure \ref{der::fig:voltage_graph_example}.

\end{example}

\section{Quantum isomorphisms}
\label{sec::quantum_isomorphisms}

In Example \ref{crossed::example:tensor_product_quantum_set} we discussed the special case of a crossed product quantum set with respect to the trivial action. In this case, one has an isomorphism of quantum sets
$$
    (\tilde{B}, \tilde{\psi}) \rtimes_{\mathrm{triv}} \hat{G} \cong (\tilde{B}, \tilde{\psi}) \otimes (C(G), \psi_G)
$$
between the crossed product quantum set and the tensor product quantum set.

Now we will show that \emph{any} crossed product quantum set is quantum isomorphic to a tensor product quantum set, i.e.
\begin{align*}
    (\tilde{B}, \tilde{\psi}) \rtimes_{\hat{\alpha}} \hat{G} \cong_q (\tilde{B}, \tilde{\psi}) \otimes (C(G), \psi_G).
\end{align*}
Further, we will show that the derived quantum graph $\tilde{A} \rtimes_{\hat{\alpha}} \hat{G}$ of a voltage quantum graph $\tilde{A} = (\tilde{A}_g)_{g \in G}$ is quantum isomorphic to the derived quantum graph $\tilde{A} \rtimes_{\mathrm{triv}} \hat{G}$ with respect to the trivial action, i.e.
\begin{align*}
    \tilde{A} \rtimes_{\hat{\alpha}} \hat{G} \cong_q \tilde{A} \rtimes_{\mathrm{triv}} \hat{G}.
\end{align*}
Note that the quantum graph on the right-hand side is defined on the tensor product quantum set $(\tilde{B}, \tilde{\psi}) \otimes (C(G), \psi_G)$.

As a consequence, in Corollary \ref{qiso::cor:quantum_isomorphism_between_classical_derived_graph_and_quantum_graph} we obtain a technique to construct quantum isomorphic quantum graphs for many classical graphs .
How does that work? Assume a classical graph $A$ can be written as $\tilde{A} \rtimes_{\mathrm{triv}} \hat{G}$ for some voltage quantum graph $\tilde{A} = (\tilde{A}_g)_{g \in G}$. If $\tilde{A}$ is invariant under a non-trivial action $\hat{\alpha}$ of $\hat{G}$, then $A$ is quantum isomorphic to the derived quantum graph $\tilde{A} \rtimes_{\hat{\alpha}} \hat{G}$, which is usually a non-classical quantum graph. 

Before we state the quantum isomorphism result for quantum sets, let us recall some notation. For a (finite abelian) group $G$ the C*-algebra $C(G)$ has two canonical bases: First, the basis $(e_g)_{g \in G}$ given by the indicator functions, and second the basis $(\tilde{u}_\chi)_{\chi \in \hat{G}}$ given by the functions with $\tilde{u}_\chi(g) = \chi(g)$ for all $g \in G$. The algebra $C(G)$ becomes a quantum set when it is equipped with the functional $\psi_G$ given by, equivalently,
\begin{align*}
    \psi_G(e_g) = 1 \text{ for all } g \in G,
    \quad \text{ or } \quad
    \psi_G(\tilde{u}_\chi) = n \delta_{\chi = 1} \text{ for all } \chi \in \hat{G},
\end{align*}
where $n = |G| = |\hat{G}|$ and $1 \in \hat{G}$ is the trivial character, see also Example \ref{pre::ex:examples_of_quantum_sets}. Further, the functional $\psi_G$ induces an inner product on $C(G)$ which makes the latter into a Hilbert space denoted by $L^2(G)$. The elements of $B(L^2(G))$ can be identified with $\hat{G}$-index matrices, and a basis for these is given by the standard matrix units $E_{\xi, \zeta}$ for $\xi, \zeta \in \hat{G}$.
To make the notation better readable we write $e_\chi$ for $\tilde{u}_\chi \in C(G)$ when we consider it as an element of the Hilbert space $L^2(G)$. Yet we remind the reader that $\langle e_\chi, e_\xi \rangle = \psi_G(\tilde{u}_\chi^\ast \tilde{u}_\xi) = n \delta_{\chi = \xi}$ for all $\chi, \xi \in \hat{G}$.

\begin{proposition}
    \label{land::prop:quantum_isomorphism_between_crossed_product_and_tensor_product}
    Let $\hat{\alpha}: \hat{G} \to \mathrm{Aut}(\tilde{B}, \tilde{\psi})$ be an action of the dual of a finite abelian group $\hat{G}$ on a quantum set $(\tilde{B}, \tilde{\psi})$, and let $\mathcal{H} := L^2(G)$.
    Then, the map
    \begin{align*}
        \rho: \left\{ \;
        \begin{aligned}   
            \tilde{B} \otimes C(G) &\to \left(\tilde{B} \rtimes_{\hat{\alpha}} \hat{G}\right) \otimes B(\mathcal{H}), \\
            b \otimes \tilde{u}_\chi &\mapsto \sum_{\zeta \in \hat{G}} u_\zeta b u_\zeta^\ast u_\chi \otimes E_{\zeta, \chi^{-1} \zeta}
        \end{aligned}\right.
    \end{align*}
    is a quantum isomorphism
    \begin{align*}
        (\tilde{B}, \tilde{\psi}) \rtimes_{\hat{\alpha}} \hat{G} \cong_q (\tilde{B}, \tilde{\psi}) \otimes (C(G), \psi_G).
    \end{align*}
\end{proposition}

\begin{proof}
    To prove that $\rho$ is a $\ast$-homomorphism, we construct it via the universal property of the tensor product.
    Let $(\hat{u}_\chi)_{\chi \in \hat{G}}$ be the natural representation of $\hat{G}$ on $\mathcal{H}$ given by 
    \begin{align*}
        \hat{u}_\chi e_\xi = e_{\chi \xi} \quad \text{ for all } \chi, \xi \in \hat{G},
    \end{align*}
    and consider the maps
    \begin{align*}
        \rho_1: \tilde{B} &\to \left(\tilde{B} \rtimes_{\hat{\alpha}} \hat{G}\right) \otimes B(\mathcal{H}), & b &\mapsto \sum_{\chi \in \hat{G}} u_\chi b u_\chi^\ast \otimes E_{\chi, \chi}, \\
        \rho_2: C(G) &\to \left(\tilde{B} \rtimes_{\hat{\alpha}} \hat{G}\right) \otimes B(\mathcal{H}), & \tilde{u}_\chi &\mapsto u_\chi \otimes \hat{u}_\chi.
    \end{align*}
    It is not hard to see that $\rho_1$ and $\rho_2$ are unital $\ast$-homomorphisms. Furthermore, they have commuting ranges since for all $b \in \tilde{B}$ and $\chi \in \hat{G}$,
    \begin{align*}
        \rho_2(\tilde{u}_\chi) \rho_1(b) \rho_2(\tilde{u}_\chi)^\ast 
            &= \sum_{\xi \in \hat{G}} u_\chi u_\xi b u_\xi^\ast u_\chi^\ast \otimes \hat{u}_\chi E_{\xi, \xi} \hat{u}_\chi^\ast \\
            &= \sum_{\zeta \in \hat{G}} u_\zeta b u_\zeta^\ast \otimes E_{\zeta, \zeta}
            = \rho_1(b).
    \end{align*}
    Therefore, the universal property of the tensor product of C*-algebras yields a unique unital $\ast$-homomorphism
    \begin{align*}
        \rho: \tilde{B} \otimes C(G) \to \left(\tilde{B} \rtimes_{\hat{\alpha}} \hat{G}\right) \otimes B(\mathcal{H}),
    \end{align*}
    such that $\rho|_{\tilde{B}} = \rho_1$ and $\rho|_{C(G)} = \rho_2$. A short calculation confirms that this $\ast$-homomorphism is the map $\rho$ from the statement.

    It remains to show that the associated map 
    \begin{align*}
        p: L^2(\tilde{B} \otimes C(G)) \otimes \mathcal{H} \to L^2(\tilde{B} \rtimes_{\hat{\alpha}} \hat{G}) \otimes \mathcal{H}
    \end{align*}
    from Definition \ref{pre::def:qisomoprhisms_of_qsets_and_qgraphs} is unitary. First, let us show that the adjoint of $p$ is given by
    \begin{align}
        \label{qiso::eq:formula_for_p_adjoint}
        p^\ast: \left\{
        \begin{aligned} 
            L^2(\tilde{B} \rtimes_{\hat{\alpha}} \hat{G}) \otimes \mathcal{H} &\to L^2(\tilde{B} \otimes C(G)) \otimes \mathcal{H}, \\
            b_\chi u_\chi \otimes e_\xi &\mapsto u_\xi^\ast b_\chi u_\xi \otimes \tilde{u}_\chi \otimes e_{\chi^{-1}\xi}.
        \end{aligned} \right.
    \end{align}
    Indeed, for all $c \in \tilde{B}$ and $\alpha, \beta \in \hat{G}$ we have
    \begin{align*}
        p(c \otimes \tilde{u}_\alpha \otimes e_\beta) 
        &= \sum_{\gamma \in \hat{G}} u_\gamma c u_\gamma^\ast u_\alpha \otimes E_{\gamma, \alpha^{-1} \gamma} e_\beta
        = u_{\alpha\beta} c u_{\beta}^\ast \otimes e_{\alpha\beta}.
    \end{align*}
    It follows for all $b_\chi \in \tilde{B}$ and $\chi, \xi \in \hat{G}$,
    \begin{align*}
        \begin{aligned}
        &\hspace{-.5cm} \langle b_\chi u_\chi \otimes e_\xi, p(c \otimes \tilde{u}_\alpha \otimes e_\beta) \rangle_{\psi_{\rtimes} \otimes \psi_G} \\
            &= \langle b_\chi u_\chi \otimes e_\xi,  u_{\alpha\beta } c u_{\beta}^\ast \otimes e_{\alpha\beta} \rangle_{\psi_{\rtimes} \otimes \psi_G} \\
            &= n \delta_{\xi=\alpha\beta} \langle b_\chi u_\chi, u_{\alpha\beta } c u_{\beta}^\ast \rangle_{\psi_\rtimes}.
    \end{aligned}
    \end{align*}
    Using formula \eqref{land::eq:formula_for_scalar_product_on_crossed_product_algebra} for $\langle \cdot, \cdot \rangle_{\psi_\rtimes}$, the previous expression becomes
    \begin{align}
        \label{land::eq:computation_of_p_to_compare_with_p_adjoint}
        \begin{aligned}
            \langle b_\chi u_\chi \otimes e_\xi, p(c \otimes \tilde{u}_\alpha \otimes e_\beta) \rangle_{\psi_{\rtimes} \otimes \psi_G}
            &= n^2 \delta_{\xi=\alpha\beta} \delta_{\chi=\alpha} \langle b_\chi, \hat{\alpha}_{\alpha\beta}(c) \rangle_{\tilde{\psi}}.
        \end{aligned}
    \end{align}
    On the other hand, for the map $p^\ast$ from \eqref{qiso::eq:formula_for_p_adjoint} we compute
    \begin{align*}
        \begin{aligned}
        \langle p^\ast(b_\chi u_\chi \otimes e_\xi), &c \otimes \tilde{u}_\alpha \otimes e_\beta \rangle_{\tilde{\psi} \otimes \psi_G \otimes \psi_G} \\
            &= \langle u_\xi^\ast b_\chi u_\xi \otimes \tilde{u}_\chi \otimes e_{\chi^{-1}\xi}, c \otimes \tilde{u}_\alpha \otimes e_\beta \rangle_{\tilde{\psi} \otimes \psi_G \otimes \psi_G} \\
            &= \langle u_\xi^\ast b_\chi u_\xi, c \rangle_{\tilde{\psi}} \langle \tilde{u}_\chi, \tilde{u}_\alpha \rangle_{\psi_G} \langle e_{\chi^{-1}\xi}, e_\beta \rangle_{\psi_G} \\
            &= n^2 \delta_{\chi=\alpha} \delta_{\chi^{-1} \xi = \beta} \langle u_\xi^\ast b_\chi u_\xi, c \rangle_{\tilde{\psi}}.
        \end{aligned}
    \end{align*}
    Using $\tilde{\psi} \circ \hat{\alpha}_\xi = \tilde{\psi}$ it is not hard to see
    \begin{align*}
        \langle u_\xi^\ast b_\chi u_\xi, c \rangle_{\tilde{\psi}} 
            = \tilde{\psi}(u_\xi^\ast b_\chi^\ast u_\xi c)
            = \tilde{\psi}(b_\chi^\ast u_\xi c u_\xi^\ast)
            = \langle b_\chi, \hat{\alpha}_\xi(c) \rangle_{\tilde{\psi}}.
    \end{align*}
    Since $\chi^{-1} \xi = \beta$ and $\chi=\alpha$ entail $\xi = \alpha \beta$ we arrive at
    \begin{align}
        \langle p^\ast(b_\chi u_\chi \otimes e_\xi), c \otimes \tilde{u}_\alpha \otimes e_\beta \rangle_{\tilde{\psi} \otimes \psi_G \otimes \psi_G}
            = n^2 \delta_{\chi=\alpha} \delta_{\xi = \alpha \beta} \langle b_\chi, \hat{\alpha}_{\alpha \beta}(c) \rangle_{\tilde{\psi}}.
    \end{align}
    This is the same expression as in \eqref{land::eq:computation_of_p_to_compare_with_p_adjoint}; consequently $p^\ast$ is indeed the adjoint of the map $p$.

    Finally, we are ready to show that $p$ is a unitary map. It suffices to prove $p p^\ast = \mathrm{Id}$. For this, let $b_\chi u_\chi \in \tilde{B} \rtimes_{\hat{\alpha}} \hat{G}$ and $\xi \in \hat{G}$ and observe
    \begin{align*}
        pp^\ast(b_\chi u_\chi \otimes e_\xi) 
            &= p\left( u_\xi^\ast b_\chi u_\xi \otimes \tilde{u}_\chi \otimes e_{\chi^{-1}\xi} \right) \\
            &=  \sum_{\gamma \in \hat{G}} u_\gamma u_\xi^\ast b_\chi u_\xi u_\gamma^\ast u_\chi \otimes E_{\gamma, \chi^{-1} \gamma} e_{\chi^{-1}\xi} \\
            &= b_\chi u_\chi \otimes e_\xi.
    \end{align*}
    This finishes the proof.
\end{proof}

\begin{thm}
    \label{qiso::thm:quantum_isomorphism_between_derived_graphs}
    Let $(\tilde{B}, \tilde{\psi})$, $\hat{\alpha}: \hat{G} \to \mathrm{Aut}(\tilde{B}, \tilde{\psi})$ and $(\tilde{B} \rtimes_{\hat{\alpha}} \hat{G}, \psi_\rtimes)$ be as above. Further, let $\tilde{A} = (\tilde{A}_g)_{g \in G}$ be a voltage quantum graph on $(\tilde{B}, \tilde{\psi})$ with respect to the action $\hat{\alpha}$ of $\hat{G}$.
    Then we have a quantum isomorphism
    \begin{align*}
        \tilde{A} \rtimes_{\hat{\alpha}} \hat{G} \cong_q \tilde{A} \rtimes_{\mathrm{triv}} G
    \end{align*}
    which is given by the map $\rho: \tilde{B} \otimes C(G) \to (\tilde{B} \rtimes_{\hat{\alpha}} \hat{G}) \otimes B(\mathcal{H})$ from Proposition \ref{land::prop:quantum_isomorphism_between_crossed_product_and_tensor_product}.
\end{thm}

\begin{proof}
    We need to show that
    \begin{align*}
        (\tilde{A} \rtimes_{\hat{\alpha}} \hat{G}) \rho
            = \rho (\tilde{A} \rtimes_{\mathrm{triv}} \hat{G})
    \end{align*}
    holds where the left-hand side is to be interpreted as in Notation \ref{pre::notation:quantum_functions}. First, recall
    \begin{align*}
        \tilde{A} \rtimes_{\hat{\alpha}} \hat{G}
            &= \sum_{g \in G} m_\rtimes(X_g^\rtimes \otimes E_\rtimes^\ast \tilde{A}_g E_\rtimes) m_\rtimes^\ast \\
        \tilde{A} \rtimes_{\mathrm{triv}} \hat{G}
            &= \sum_{g \in G} m_\otimes(X_g^\otimes \otimes E_\otimes^\ast \tilde{A}_g E_\otimes) m_\otimes^\ast,
    \end{align*}
    where we distinguish the structure maps $m, E, X_g$ on the crossed product quantum sets $\tilde{B} \rtimes_{\hat{\alpha}} \hat{G}$ and $\tilde{B} \otimes C(G) \cong \tilde{B} \rtimes_{\mathrm{triv}} \hat{G}$ by super-/subscripts $\rtimes$ and $\otimes$, respectively. 
    We will show for any fixed $g \in G$ that
    \begin{align*}
        (m_\rtimes(X_g^\rtimes \otimes E_\rtimes^\ast \tilde{A}_g E_\rtimes) m_\rtimes^\ast) \rho
        = \rho (m_\otimes(X_g^\otimes \otimes E_\otimes^\ast \tilde{A}_g E_\otimes) m_\otimes^\ast)
    \end{align*}
    holds. From Proposition \ref{pre::prop:properties_of_qisomorphisms} we know the equalities
    \begin{align*}
        \rho m_\otimes = m_\rtimes (\rho \otimes \rho)
        \quad \text{ and } \quad
        m_\rtimes^\ast \rho = (\rho \otimes \rho) m_\otimes^\ast.
    \end{align*}
    Thus, all we need to do is to prove
    \begin{align*}
        X_g^\rtimes \rho = \rho X_g^\otimes
        \quad \text{ and } \quad
        E_\rtimes^\ast \tilde{A}_g E_\rtimes \rho = \rho (E_\otimes^\ast \tilde{A}_g E_\otimes).
    \end{align*}
    To show the first equality we need to consider for $b \in \tilde{B}$ and $\xi \in \hat{G}$,
    \begin{align*}
        X_g^\rtimes \rho(b \otimes \tilde{u}_\xi)
            = \frac{1}{n} \sum_{\chi \in \hat{G}} \overline{\chi(g)} u_\chi \psi_\rtimes(u_\chi^\ast \rho(b \otimes \tilde{u}_\xi)).
    \end{align*}
    For fixed $\chi \in \hat{G}$ the last factor can be computed as
    \begin{align*}
        \psi_\rtimes(u_\chi^\ast \rho(b \otimes \tilde{u}_\xi))
            &= \sum_{\zeta \in \hat{G}} \psi_\rtimes(u_\chi^\ast u_\zeta b u_\zeta^\ast u_\xi) E_{\zeta, \chi^{-1} \zeta} \\
            &= \sum_{\zeta \in \hat{G}} n \delta_{\chi = \xi} \tilde{\psi}(b) E_{\zeta, \chi^{-1} \zeta} \\
            &= n \delta_{\chi = \xi} \tilde{\psi}(b) \hat{u}_\chi,
    \end{align*}
    where $\hat{u}_\chi \in B(L^2(G))$ is the operator from the proof of Proposition \ref{land::prop:quantum_isomorphism_between_crossed_product_and_tensor_product} given by $\hat{u}_\chi e_\zeta = e_{\chi \zeta}$ for all $\zeta \in \hat{G}$.
    It follows
    \begin{align}
        \label{qiso::eq:Xg_rho_computation}
        \begin{aligned}
        X_g^\rtimes \rho(b \otimes \tilde{u}_\xi) 
            &= \overline{\xi(g)} \tilde{\psi}(b) u_\xi \otimes \hat{u}_\xi.
        \end{aligned}
    \end{align}
    On the other hand, we compute
    \begin{align}
        \label{qiso::eq:rho_Xg_computation}
        \begin{aligned}
        \rho X_g^\otimes(b \otimes \tilde{u}_\xi) 
            &= \rho \left( \frac{1}{n} \sum_{\chi \in \hat{G}} \overline{\chi(g)} (1 \otimes \tilde{u}_\chi) (\tilde{\psi} \otimes \psi_G)(b \otimes \tilde{u}_\chi^\ast \tilde{u}_\xi)  \right) \\
            &=  \rho \left( \overline{\xi(g)} \tilde{\psi}(b) (1 \otimes \tilde{u}_\xi) \right) \\
            &= \overline{\xi(g)} \tilde{\psi}(b) u_\xi \otimes \hat{u}_\xi.
        \end{aligned}
    \end{align}
    Comparing \eqref{qiso::eq:Xg_rho_computation} and \eqref{qiso::eq:rho_Xg_computation} shows the first desired equality.

    It remains to prove the second equality, i.e. $E_\rtimes^\ast \tilde{A}_g E_\rtimes \rho = \rho (E_\otimes^\ast \tilde{A}_g E_\otimes)$. Let us consider again fixed $b \in \tilde{B}$ and $\xi \in \hat{G}$. We have on the one hand
    \begin{align}
        \label{qiso::eq:E*.tildeAg.E.rho_computation}
        \begin{aligned}
        E_\rtimes^\ast \tilde{A}_g E_\rtimes \rho(b \otimes \tilde{u}_\xi)
            &= \sum_{\zeta \in \hat{G}} E_\rtimes^\ast \tilde{A}_g E_\rtimes (u_\zeta b u_\zeta^\ast u_\xi \otimes E_{\zeta, \xi^{-1} \zeta}) \\
            &= n \delta_{\xi = 1} \sum_{\zeta \in \hat{G}} \tilde{A}_g(u_\zeta b u_\zeta^\ast) \otimes E_{\zeta, \zeta} \\
            &= n \delta_{\xi = 1} \sum_{\zeta \in \hat{G}} u_\zeta \tilde{A}_g(b) u_\zeta^\ast \otimes E_{\zeta, \zeta},
        \end{aligned}
    \end{align}
    (where we use in final step that $\tilde{A}_g$ commutes with the action $\hat{\alpha}$),
    and on the other hand
    \begin{align}
        \label{qiso::eq:rho.E*.tildeAg.E_computation}
        \begin{aligned}
        \rho (E_\otimes^\ast \tilde{A}_g E_\otimes)(b \otimes \tilde{u}_\xi)
            &= \rho \left( n \delta_{\xi = 1} \tilde{A}_g(b) \otimes \tilde{u}_1 \right) \\
            &= n \delta_{\xi = 1} \sum_{\zeta \in \hat{G}} u_\zeta \tilde{A}_g(b) u_\zeta^\ast \otimes E_{\zeta, \zeta}.
        \end{aligned}
    \end{align}
    A comparison of \eqref{qiso::eq:E*.tildeAg.E.rho_computation} and \eqref{qiso::eq:rho.E*.tildeAg.E_computation} finishes the proof.
\end{proof}

If the derived quantum graph $\tilde{A} \rtimes_{\mathrm{triv}} \hat{G}$ from Theorem \ref{qiso::thm:quantum_isomorphism_between_derived_graphs} is a classical graph, then the previous theorem yields a method to produce quantum versions of this given classical graph, i.e. quantum graphs that are quantum isomorphic to the original classical graph.

\begin{corollary}
    \label{qiso::cor:quantum_isomorphism_between_classical_derived_graph_and_quantum_graph}
    Assume that $\Gamma^\lambda$ is a classical derived graph of a voltage graph $(\Gamma, \lambda)$ on a vertex set $V$ with respect to a finite abelian group $G$. Then for any action $\hat{\alpha}: \hat{G} \to \mathrm{Aut}(\Gamma, \lambda)$ of the dual group on the voltage graph we obtain a quantum isomorphism
    \begin{align*}
        \Gamma^\lambda \cong_q \Gamma \rtimes_{\hat{\alpha}} \hat{G},
    \end{align*}
    where we interpret $\Gamma$ on the right-hand side as a voltage quantum graph on $(\C^V, \psi_V)$ as in Example \ref{der::ex:classical_voltage_graphs_as_quantum_voltage_graphs}.

    Here with an action of $\hat{G}$ on $(\Gamma, \lambda)$ we mean an action $\hat{\alpha}: \hat{G} \to \mathrm{Aut}(V)$ which leaves the set of $\chi$-colored edges in $\Gamma$ invariant for all $\chi \in \hat{G}$.
\end{corollary}

\section{Quantum Gross--Tucker Theorem}
\label{sec::qgross_tucker}

Which quantum graphs are derived quantum graphs of a voltage quantum graph as in Definition \ref{der::def:derived_quantum_graph}?
In the classical situation, Gross and Tucker proved that whenever a graph $\Gamma$ admits a free action $\alpha: G \to \mathrm{Aut}(\Gamma)$ of a finite group $G$, then there is a suitable $G$-labeling $\lambda$ of the quotient graph $\Gamma / G$ such that
\begin{align*}
    \Gamma \cong (\Gamma / G)^\lambda,
\end{align*}
i.e. $\Gamma$ is the derived graph of the voltage graph $(\Gamma / G, \lambda)$ \cite[Theorem 2.2.2]{gross_topological_1987}.
In this section, we prove a quantum version of this result.

Thus, the setting is the following: We have given a quantum graph $A: B \to B$ on a quantum set $(B, \psi)$ together with an action $\alpha: G \to \mathrm{Aut}(A)$ of a finite abelian group $G$ on the quantum graph $A$. This means that $\alpha$ is an action of $G$ on the C*-algebra $B$ which satisfies
\begin{align*}
    \psi = \psi \alpha_g 
    \quad \text{ and } \quad
    \alpha_g A = A \alpha_g
\end{align*}
for all $g \in G$.

Where Gross and Tucker asked that the action on the graph be free, we assume that the action $\alpha$ satisfies the assumptions of Landstad's theorem including the additional condition for the quantum set version from Proposition \ref{pre::prop:landstad_for_quantum_sets}. Thus, we ask that there is a unitary representation $(u_\chi)_{\chi \in \hat{G}} \subset B$ of the dual group $\hat{G}$ inside $B$ such that
\begin{align*}
    \alpha_g(u_\chi) = \chi(g) u_\chi
    \quad \text{ and } \quad
    \psi = \psi \mathrm{Ad}(u_\chi)
\end{align*}
hold for all $g \in G$ and $\chi \in \hat{G}$. Furthermore, also the quantum graph $A$ must be invariant under $\mathrm{Ad}(u_\chi)$, i.e. we will also require
\begin{align}
    \label{gtt::eq:A_is_invariant_under_Ad_u_chi}
    \mathrm{Ad}(u_\chi) A = A \mathrm{Ad}(u_\chi).
\end{align}

Then, in view of Proposition \ref{pre::prop:landstad_for_quantum_sets} the quantum set $(B, \psi)$ is a crossed product quantum set
\begin{align*}
    (B, \psi) = (B^\alpha, \tilde{\psi}) \rtimes_{\hat{\alpha}} \hat{G}.
\end{align*}
In particular, 
\begin{itemize}
    \item $B = B^\alpha \rtimes_{\hat{\alpha}} \hat{G}$,
    \item every element of $B$ can be written uniquely as $b = \sum_{\chi \in \hat{G}} b_\chi u_\chi$ for suitable $b_\chi \in B^\alpha$,
    \item $\psi(b_\chi u_\chi) = \delta_{\chi=1} \psi(b_\chi) = n \delta_{\chi=1} \tilde{\psi}(b_\chi)$ holds for all $\chi \in \hat{G}$ and $b_\chi \in B^\alpha$ with $n = |G|$.
\end{itemize}

Note that we consistently use the notation from Notation \ref{pre::notation:crossed_product_quantum_sets} for crossed product quantum sets with $B^\alpha$ in the place of $\tilde{B}$ as well as the notation from Landstad's theorem (Theorem \ref{pre::thm:landstads_thm}) and Proposition \ref{pre::prop:landstad_for_quantum_sets}.

\begin{thm}
    \label{gtt::thm:quantum_gross_tucker_theorem}
    Let $A: B \to B$ be a quantum graph on the quantum set $(B, \psi)$ and let $\alpha: G \to \mathrm{Aut}(A)$ be an action of a finite abelian group $G$ on $A$ which satisfies the assumptions of Landstad's theorem for quantum sets (Proposition \ref{pre::prop:landstad_for_quantum_sets}) as well as the additional assumption \eqref{gtt::eq:A_is_invariant_under_Ad_u_chi}. Then the following are true:
    \begin{enumerate}
        \item The family $\tilde{A} := (\tilde{A}_g)_{g \in G}$ with
        \begin{align*}
            \tilde{A}_g = m^\ast(X_g \otimes A) m|_{B^\alpha}
        \end{align*}
        is a voltage quantum graph on $(B^\alpha, \tilde{\psi})$ with respect to the group action $\hat{\alpha}$, where $X_g = \frac{1}{n} \sum_{\chi \in \hat{G}} \overline{\chi(g)} u_\chi u_\chi^\dagger$ with $n = |G|$ is the quantum adjacency matrix from Definition \ref{der::def:derived_quantum_graph} and Proposition \ref{der::prop:derived_graph_is_quantum_graph}. 
        
        \item One has
        \begin{align*}
            A \cong \tilde{A} \rtimes_{\hat{\alpha}} \hat{G},
        \end{align*}
        i.e. $A$ is recovered as the derived graph of $\tilde{A}$ as in Definition \ref{der::def:derived_quantum_graph}.
    \end{enumerate}
\end{thm}

\begin{remark}
    \begin{enumerate}
        \item The claim that $\tilde{A}$ is a voltage quantum graph entails that every singly map $\tilde{A}_g$ maps into $B^\alpha \subset B$. This is not immediately clear since the maps $m$ and $m^\ast$ are defined on the larger algebra $B$.
        \item Let us remark again that the operation
        \begin{align*}
            (A_1, A_2) \mapsto m(A_1 \otimes A_2) m^\ast
        \end{align*}
        intuitively means to take the intersection of quantum graphs as discussed in Remark \ref{der::rem:intersection_of_qgraphs}.
    \end{enumerate}
\end{remark}

Before we prove the theorem, we will investigate some properties of the maps $\tilde{A}_g$. First, we pick up the issue mentioned in the previous remark and show that the maps $\tilde{A}_g$ are well-defined as maps from $B^\alpha$ into $B^\alpha$.

\begin{proposition}
    \label{prop:Ag_well_defined}
    The maps $\tilde{A}_g$ are well-defined, i.e. we have for all $g \in G$,
    \begin{align*}
        m(X_g \otimes A) m^\ast(B^\alpha) \subset B^\alpha.
    \end{align*}
\end{proposition}

\begin{proof}
    It suffices to show $m(X_g \otimes A)m^\ast \alpha_h = \alpha_h m(X_g \otimes A)m^\ast$ for all $h \in G$. Evidently, $m(\alpha_h \otimes \alpha_h) = \alpha_h m$ holds since $\alpha_h$ is a $\ast$-automorphism. Further, one readily verifies $\alpha_h^\ast = \alpha_{h^{-1}}$ with respect to the inner product induced by $\psi$. By taking the adjoint, the previous identity entails $(\alpha_h \otimes \alpha_h) m^\ast = m^\ast \alpha_h$ for all $h \in G$. 
    Further, $A \alpha_h = \alpha_h A$ is true since $\alpha_h \in \mathrm{Aut}(A)$ by assumption.

    It only remains to show $X_g \alpha_h = \alpha_h X_g$ for all $h \in G$.
    A quick calculation reveals
    \begin{align*}
        u_\chi^\dagger \alpha_h = \chi(h) u_\chi^\dagger
    \end{align*}
    for all $\chi \in \hat{G}$ and $h \in G$. (Recall $u_\chi^\dagger = \psi(u_\chi^\ast \cdot)$). Consequently, we have
    \begin{align*}
        X_g \alpha_h 
            = \frac{1}{n} \sum_{\chi \in \hat{G}} \overline{\chi(g)} u_\chi u_\chi^\dagger \alpha_h
            = \frac{1}{n} \sum_{\chi \in \hat{G}} \overline{\chi(g)} \chi(h) u_\chi u_\chi^\dagger
            = \alpha_h X_g,
    \end{align*}
    since $\alpha_h(u_\chi) = \chi(h) u_\chi$ is true for all $\chi \in \hat{G}$.
    Altogether, we obtain the identity $m(X_g \otimes A)m^\ast \alpha_h = \alpha_h m(X_g \otimes A)m^\ast$ as desired.
\end{proof}

The next lemma establishes a connection between the maps $\tilde{A}_g$ and a family of maps $\tilde{A}_\chi: B^\alpha \to B^\alpha$ which looks like the equivalent of the $\tilde{A}_g$ under a kind of Fourier transform. These maps have convenient algebraic properties and yield an alternative definition of the maps $\tilde{A}_g$.

\begin{lemma}
    \label{lemma:technical_properties_of_Ag_Xg}
    There are linear maps $\tilde{A}_\chi: B^\alpha \to B^\alpha$ such that
    \begin{align}
        \label{eq:definition_of_tildeA_chi}
        A\left( \sum_{\chi \in \hat{G}} b u_\chi \right) 
            = \sum_{\chi \in \hat{G}} \tilde{A}_\chi(b) u_\chi
    \end{align}
    is true for all $b \in B^\alpha$, $\chi \in \hat{G}$.
    Then we have for all $b \in B^\alpha$, $g \in G$ and $\chi, \zeta \in \hat{G}$:
    \begin{align}
        \tilde{A}_\chi \hat{\alpha}_\zeta
            &= \hat{\alpha}_\zeta \tilde{A}_\chi,
            \label{eq:Ag_commutes_with_hatalpha} \\
        \tilde{A}_\chi(b)
            &= \frac{1}{n} \sum_{\zeta \in \hat{G}} \tilde{A}_\zeta\left(b_i^{(1)}\right) \tilde{A}_{\zeta^{-1} \chi}\left(b_i^{(2)}\right),
            \label{eq:convolution_formula_for_A_chi} \\
        \tilde{A}_\chi(b)^\ast
            &= \tilde{A}_{\chi^{-1}}(b^\ast),
            \label{eq:A_chi_and_the_star} \\
        \tilde{A}_g(b)
            &= \frac{1}{n} \sum_{\zeta \in \hat{G}} \zeta(g) \tilde{A}_\zeta(b).
            \label{eq:formula_for_tildeA}
    \end{align}
    where in the second line $\tilde{m}^\ast(b) = \sum_i b_i^{(1)} \otimes b_i^{(2)}$ and we sum implicitly over $i$.
\end{lemma}

\begin{proof}
    Recall that $\alpha$ respects the quantum adjacency matrix $A$ in the sense that $A \alpha_g = \alpha_g A$ holds for all $g \in G$. It follows immediately that $A$ preserves the subalgebras $B^\alpha u_\chi$, i.e. 
    \begin{align*}
        A(B^\alpha u_\chi) \subset B^\alpha u_\chi
    \end{align*}
    holds for all $\chi \in \hat{G}$. Therefore, the maps $\tilde{A}_\chi$ from \eqref{eq:definition_of_tildeA_chi} exist.

    Ad \eqref{eq:Ag_commutes_with_hatalpha}.
    Using the definition of $\tilde{A}_\chi$ and $A \hat{\alpha}_\zeta = \hat{\alpha}_\zeta A$ one calculates for all $b \in B^\alpha$
    \begin{align*}
        \tilde{A}_\chi\left(\hat{\alpha}_\zeta(b)\right)
            = A(\hat{\alpha}_\zeta(b) u_\chi) u_{\chi^{-1}}
            = \hat{\alpha}_\zeta\left(A(b u_\chi)\right) u_{\chi^{-1}}
            = \hat{\alpha}_\zeta\left(\tilde{A}_\chi(b)\right).
    \end{align*}
    
    Ad \eqref{eq:convolution_formula_for_A_chi}.
    Since $A$ is a quantum adjacency matrix, we have $m(A \otimes A)m^\ast = A$. Using the formula \eqref{land::eq:comultiplication_formula_on_crossed_product} for the comultiplication on $B$ we get for all $b \in B^\alpha, \chi \in \hat{G}$,
    \begin{align*}
        A(b u_\chi)
            &= \frac{1}{n} \sum_{\zeta \in \hat{G}} A\left( b_i^{(1)} u_\zeta \right) A\left( \hat{\alpha}_{\zeta^{-1}}\left(b_i^{(2)}\right) u_{\zeta^{-1} \chi} \right) \\
            &= \frac{1}{n} \sum_{\zeta \in \hat{G}} \tilde{A}_\zeta\left(b_i^{(1)}\right)\tilde{\alpha}_{\zeta}\left(\tilde{A}_{\zeta^{-1} \chi}\left(\hat{\alpha}_{\zeta^{-1}}\left(b_i^{(2)}\right)\right)\right) u_\zeta u_{\zeta^{-1} \chi} \\
            &= \frac{1}{n} \sum_{\zeta \in \hat{G}} \tilde{A}_\zeta\left(b_i^{(1)}\right)\tilde{A}_{\zeta^{-1} \chi}\left(b_i^{(2)}\right) u_\chi,
    \end{align*}
    where we use $\tilde{A}_{\zeta^{-1}\chi} \hat{\alpha}_{\zeta^{-1}} = \hat{\alpha}_{\zeta^{-1}} \tilde{A}_{\zeta^{-1}\chi}$ in the second step. Comparing this with \eqref{eq:definition_of_tildeA_chi} yields the desired formula.

    Ad \eqref{eq:A_chi_and_the_star}.
    Since $A$ is a quantum adjacency matrix it is $\ast$-preserving in the sense that $A\left(\left(b u_\chi\right)^\ast\right) = A(b u_\chi)^\ast$ holds for all $b \in B^\alpha$ and $\chi \in \hat{G}$. With the definition of $\tilde{A}_\chi$ it follows
    \begin{align*}
        \tilde{A}_{\chi^{-1}}(\hat{\alpha}_{\chi^{-1}}(b^\ast)) u_{\chi^{-1}}
            &= A(\hat{\alpha}_{\chi^{-1}}(b^\ast) u_{\chi^{-1}}) 
            = A\left((b u_\chi)^\ast\right) \\
            &= A(b u_\chi)^\ast 
            = \left( \tilde{A}_\chi(b) u_\chi \right)^\ast 
            = \hat{\alpha}_{\chi^{-1}}\left(\tilde{A}_\chi(b)^\ast\right) u_{\chi^{-1}}.
    \end{align*}
    The claim follows by using $\tilde{A}_{\chi^{-1}} \hat{\alpha}_{\chi^{-1}} = \hat{\alpha}_{\chi^{-1}} \tilde{A}_{\chi^{-1}}$ and applying $\hat{\alpha}_\chi$ on both sides.

    Ad \eqref{eq:formula_for_tildeA}. 
    Let us use the definition of $\tilde{A}_g$ together with the formula \eqref{land::eq:comultiplication_formula_on_crossed_product} for the comultiplication on $B$ and \eqref{der::eq:formula_for_X_g} to obtain
    \begin{align*}
        \tilde{A}_g(b) 
            &= m(X_g \otimes A) m^\ast(b) \\
            &= \frac{1}{n} \sum_{\zeta \in \hat{G}} X_g(b_i^{(1)} u_\zeta) A\left( \hat{\alpha}_{\zeta^{-1}}\left(b_i^{(2)}\right) u_{\zeta^{-1}}\right) \\
            &= \frac{1}{n} \sum_{\zeta \in \hat{G}} \overline{\zeta(g)} \tilde{\psi}(b_i^{(1)}) u_\zeta A\left( \hat{\alpha}_{\zeta^{-1}}\left(b_i^{(2)}\right) u_{\zeta^{-1}}\right). 
    \end{align*}
    Since $A$ commutes with the action $\hat{\alpha}$ the properties of the comultiplication on $B$ imply that this expression is equal to
    \begin{align*}
        \frac{1}{n} \sum_{\zeta \in \hat{G}} \overline{\zeta(g)} \tilde{\psi}(b_i^{(1)}) A_{\zeta^{-1}}\left( b_i^{(2)} \right)
            &= \frac{1}{n} \sum_{\zeta \in \hat{G}} \overline{\zeta(g)} \tilde{A}_{\zeta^{-1}}(b).
    \end{align*}
    The statement follows by a change of variable.
\end{proof}

We are now ready to prove the main theorem of this section which has been stated before as Theorem \ref{gtt::thm:quantum_gross_tucker_theorem}.

\begin{proof}[Proof of Theorem \ref{gtt::thm:quantum_gross_tucker_theorem}]
    We first show that each $\tilde{A}_g$ is a quantum adjacency matrix on $(B^\alpha, \tilde{\psi})$ which is invariant under the action $\hat{\alpha}$.
    To start with, let us verify that $\tilde{A}_g$ is Schur-idempotent. Using formula \eqref{eq:formula_for_tildeA} from the previous lemma, we get for all $b \in B^\alpha$,
    \begin{align*}
        \tilde{m}\left(\tilde{A}_g \otimes \tilde{A}_g\right) \tilde{m}^\ast(b)
            &= \frac{1}{n^2} \sum_{\chi, \xi \in \hat{G}} \chi(g) A_\chi(b_i^{(1)}) \xi(g) A_\xi(b_i^{(2)}) \\
            &= \frac{1}{n^2} \sum_{\eta \in \hat{G}} \eta(g) \sum_{\chi, \xi: \chi \xi = \eta} A_\chi(b_i^{(1)}) A_\xi(b_i^{(2)}).
    \end{align*}
    By \eqref{eq:convolution_formula_for_A_chi} we have $\frac{1}{n} \sum_{\chi, \xi: \chi \xi = \eta} A_\chi(b_i^{(1)}) A_\xi(b_i^{(2)}) = A_\eta(b)$ for all $\eta \in \hat{G}$. Thus, the above is equal to
    \begin{align*}
        \frac{1}{n} \sum_{\eta \in \hat{G}} \eta(g) A_\eta(b) = \tilde{A}_g(b).
    \end{align*} 
    Next, we show that $\tilde{A}_g$ is $\ast$-preserving. Using \eqref{eq:A_chi_and_the_star} and the formula \eqref{eq:formula_for_tildeA} from the previous lemma one gets for all $b \in B^\alpha$ and $g \in G$,
    \begin{align*}
        \tilde{A}_g(b^\ast)  
            &= \frac{1}{n} \sum_{\zeta \in \hat{G}} \zeta(g) \tilde{A}_\zeta(b^\ast) \\
            &= \frac{1}{n} \sum_{\zeta \in \hat{G}} \zeta(g) \tilde{A}_{\zeta^{-1}}(b)^\ast \\
            &= \left( \frac{1}{n} \sum_{\zeta \in \hat{G}} \zeta^{-1}(g) \tilde{A}_{\zeta^{-1}}(b) \right)^\ast \\
            &= \tilde{A}_g(b)^\ast,
    \end{align*}
    as desired. Altogether, it follows that $\tilde{A}_g$ is a quantum adjacency matrix. For $\tilde{A}$ being a voltage graph we also need $\tilde{A}_g \hat{\alpha}_\chi = \hat{\alpha}_\chi \tilde{A}_g$ for all $\chi \in \hat{G}$. However, this follows immediately by combining \eqref{eq:Ag_commutes_with_hatalpha} and \eqref{eq:formula_for_tildeA}.

    It remains to prove $A = \tilde{A} \rtimes_{\hat{\alpha}} \hat{G}$, i.e.
    \begin{align*}
        A = \sum_{g \in G} m(X_{g} \otimes E_\rtimes^\ast \tilde{A}_g E_\rtimes) m^\ast
    \end{align*}
    Using the formula for the comultiplication \eqref{land::eq:comultiplication_formula_on_crossed_product} as well as \eqref{eq:formula_for_E*} one immediately obtains for all $b \in B^\alpha$ and $\chi \in \hat{G}$,
    \begin{align*}
        m(X_{g} \otimes E_\rtimes^\ast \tilde{A}_g E_\rtimes) m^\ast(b u_\chi)
            &= \frac{1}{n} \sum_{\zeta \in \hat{G}} X_g\left(b_i^{(1)} u_\zeta\right) \, n \delta_{\zeta^{-1}\chi = 1} \tilde{A}_g\left(\hat{\alpha}_{\zeta^{-1}}(b_i^{(2)})\right) \\
            &= X_g\left(b_i^{(1)} u_\chi\right) \tilde{A}_g\left(\hat{\alpha}_{\chi^{-1}}(b_i^{(2)})\right) \\
            &= \overline{\chi({g})} \tilde{\psi}\left(b_i^{(1)}\right) \tilde{A}_g\left(b_i^{(2)}\right) u_\chi,
    \end{align*}
    where we use \eqref{der::eq:formula_for_X_g} in the last step. As $\left(\tilde{\psi} \otimes \mathrm{Id}_{B^\alpha}\right) \tilde{m}^\ast = \mathrm{Id}_{B^\alpha}$, the above reduces to
    \begin{align*}
        \overline{\chi(g)} \tilde{A}_g(b) u_\chi.
    \end{align*}
    By summing over all $g \in G$ and using the formula \eqref{eq:formula_for_tildeA} for $\tilde{A}_g$ we arrive at
    \begin{align*}
        \sum_{g \in G} m^\ast(X_g \otimes E_\rtimes^\ast \tilde{A}_g E_\rtimes) m(b u_\chi)
            &= \frac{1}{n} \sum_{g \in G} \sum_{\zeta \in \hat{G}} \overline{\chi(g)} \zeta(g) A_\zeta(b) u_\chi \\
            &= A_\chi(b) u_\chi,
    \end{align*}
    for $\sum_{g \in G} \chi(g) \overline{\zeta(g)} = n \delta_{\zeta = \chi}$.
    This concludes the proof.
\end{proof}

\begin{remark}
    If $A$ is in fact a classical graph (aka a quantum graph on $(\C^V, \psi_V)$ for some finite set $V$), then an action $\alpha: G \to \mathrm{Aut}(A)$ satisfies the conditions of Landstad's theorem if and only if it is a free action. Moreover, the notorious operators $\mathrm{Ad}(u_\chi)$ on $\C^V$ are trivial, and therefore the additional conditions
    \begin{align*}
        \psi_V = \psi_V \, \mathrm{Ad}(u_\chi) 
        \quad \text{and} \quad
        A \, \mathrm{Ad}(u_\chi) = \mathrm{Ad}(u_\chi) A
    \end{align*}
    for Theorem \ref{gtt::thm:quantum_gross_tucker_theorem} are automatically satisfied. The quantum set $(B^\alpha, \tilde{\psi})$ is clearly isomorphic to $(\C^{V/G}, \psi_{V/G})$ and the induced action $\hat{\alpha}$ is trivial. Thus, in the classical case our Gross--Tucker theorem just says that a free action of a finite abelian group on a graph $A$ gives rise to a voltage graph structure on the quotient graph $A/G$ such that the derived graph is isomorphic to $A$. This recovers exactly the classical Gross--Tucker theorem.
\end{remark}

\section{Quantum graphs on \texorpdfstring{$M_2$}{M2} as derived quantum graphs}
\label{sec::examples}

Quantum graphs over the quantum space $M_2$ have been extensively studied~\cite{matsuda_classification_2022,gromada_examples_2022, kiefer_complete_2025}. Let us recall the classification of undirected loopfree graphs on $(M_2,2{\rm Tr})$, which was discovered independently by Matsuda~\cite{matsuda_classification_2022} and Gromada~\cite{gromada_examples_2022}.
\begin{proposition}[{\cite[Proposition~3.14]{gromada_examples_2022}}]\label{prop::gromada_tweaked}
    Up to isomorphism, there are three nonempty undirected loopfree quantum graphs on $(M_2,2{\rm Tr})$ whose quantum adjacency matrices read:
    \begin{enumerate}
\item $A=P_3$,
\item $A=P_3+P_2$,
\item $A=P_3+P_2+P_1$,
    \end{enumerate}
    where
  \[
P_1=\begin{pmatrix}
    0 & 0 & 0 & 1\\
    0 & 0 & 1 & 0\\
    0 & 1 & 0 & 0\\
    1 & 0 & 0 & 0
\end{pmatrix},~
P_2=\begin{pmatrix}
    0 & 0 & 0 & 1\\
    0 & 0 & -1 & 0\\
    0 & -1 & 0 & 0\\
    1 & 0 & 0 & 0
\end{pmatrix},~
P_3=\begin{pmatrix}
    1 & 0 & 0 & 0\\
    0 & -1 & 0 & 0\\
    0 & 0 & -1 & 0\\
    0 & 0 & 0 & 1
\end{pmatrix}
\]
are matrices given with respect to the standard basis of $M_2$.
\end{proposition}
\begin{remark}
It is worth noting that the quantum adjacency matrices from the points (a)--(c) in the above proposition are equal to $\sigma_3\otimes\sigma_3^*$, $\sigma_2\otimes\sigma_2^*+\sigma_3\otimes\sigma_3^*$, and $\sigma_1\otimes\sigma_1^*+\sigma_2\otimes\sigma_2^*+\sigma_3\otimes\sigma_3^*$, respectively, where
\[
\sigma_1=\begin{pmatrix}0 & 1\\1 & 0\end{pmatrix},\qquad \sigma_2=\begin{pmatrix}0 & -i\\i & 0\end{pmatrix}, \qquad \sigma_3=\begin{pmatrix}1 & 0\\0 & -1\end{pmatrix},
\]
are the Pauli matrices.
\end{remark}

In this section, we show that all three non-iso\-morphic (nonempty, undirected, loopfree) quantum graphs on $(M_2(\mathbb{C}),2{\rm Tr})$ can be obtained as derived quantum graphs of classical (directed) voltage graphs with two vertices. We emphasize that neither classical nor quantum voltage graphs need to be simple loopfree graphs, which gives us greater flexibility. In fact, our final example will include both loops and multiple edges between vertices. Throughout this section we denote the elements of $\mathbb{Z}_2$ by $\{0,1\}$ and its dual $\hat{\mathbb{Z}}_2$ by $\{\chi_0,\chi_1\}$. We warn the reader that the element $1$ is no longer the identity element.

\begin{proposition}
    \label{ex::prop:qgraphs_on_M2_as_derived_qgraphs}
    Up to isomorphism, all undirected loopfree quantum graphs on $(M_2,2{\rm Tr})$ can be obtained as derived quantum graphs of classical $\mathbb{Z}_2$-voltage graphs on two vertices. In particular, they are quantum isomorphic to classical graphs on four vertices.
\end{proposition}
\begin{proof}
First, we write $(M_2,2{\rm Tr})$ as a crossed product quantum set. To this end, consider the quantum set $(\mathbb{C}^2,\psi_2)$ as in Example \ref{pre::ex:classical_graphs_as_quantum_graphs}. There is a~natural action $\hat{\alpha}$ of $\hat{\mathbb{Z}}_2$ on $\mathbb{C}^2$ given by the automorphism
\[
\hat{\alpha}_{\chi_1}:\begin{pmatrix}z_1\\z_2\end{pmatrix}\longmapsto \begin{pmatrix}z_2\\z_1\end{pmatrix},\qquad z_1,z_2\in\mathbb{C}.
\]
Note that $\psi_2$ is invariant under this action. Recall that the crossed product algebra $\mathbb{C}^2\rtimes_{\hat{\alpha}}\hat{\mathbb{Z}}_2$ is isomorphic with $M_2$ via the isomorphism induced by the following assignments
\[
\begin{pmatrix}1\\0\end{pmatrix}\longmapsto \begin{pmatrix}1 & 0\\0 & 0\end{pmatrix},\qquad \begin{pmatrix}0\\1\end{pmatrix}\longmapsto \begin{pmatrix}0 & 0\\ 0 & 1\end{pmatrix},\qquad u_{\chi_1}\longmapsto \begin{pmatrix}0 & 1\\1 & 0\end{pmatrix}.
\]
Furthermore, $(\mathbb{C}^2\rtimes_{\hat{\alpha}}\hat{\mathbb{Z}}_2,\psi_\rtimes)$ and $(M_2(\mathbb{C}),2{\rm Tr})$ are isomorphic as quantum sets.

A classical $\mathbb{Z}_2$-voltage graph with respect to the action $\hat{\alpha}$ is nothing else but a pair $\tilde{A}=(\tilde{A}_0,\tilde{A}_1)$, where $\tilde{A}_0$ and $\tilde{A}_1$ are adjacency matrices on two vertices satisfying
\begin{equation}\label{ex::adjmat1}
\tilde{A}_0\hat{\alpha}_{\chi_1}=\hat{\alpha}_{\chi_1}\tilde{A}_0\qquad\text{and}\qquad \tilde{A}_1\hat{\alpha}_{\chi_1}=\hat{\alpha}_{\chi_1}\tilde{A}_1\,.
\end{equation}
Then the quantum adjacency matrix of the derived quantum graph on $(\mathbb{C}^2\rtimes_{\hat{\alpha}}\hat{\mathbb{Z}}_2,\psi_\rtimes)$, where $\psi_\rtimes$ is as in Definition~\ref{crossed::def:crossed_product_quantum_set}, is given by the formula
\[
A=m(X_0\otimes E^*_\rtimes \tilde{A}_0E_\rtimes)m^*+m(X_1\otimes E^*_\rtimes \tilde{A}_1E_\rtimes)m^*\\
\]
where $E_\rtimes$ is the scaled conditional expectation (see Notation~\ref{pre::notation:crossed_product_quantum_sets}) and
\[
X_0=\begin{pmatrix}1 & 0 & 0 & 1\\ 0 & 1 & 1 & 0 \\ 0 & 1 & 1 & 0\\ 1 & 0 & 0 & 1\end{pmatrix}\qquad\text{and}\qquad X_1=\begin{pmatrix}1 & 0 & 0 & 1\\ 0 & -1 & -1 & 0 \\ 0 & -1 & -1 & 0\\ 1 & 0 & 0 & 1\end{pmatrix}
\]
(see Definition~\ref{der::def:derived_quantum_graph} for details.)

To prove the proposition, let us first observe that the empty graph on $(M_2,2{\rm Tr})$ is given by the derived quantum graph of a graph with two vertices and no edges. Next, the three quantum graphs from Proposition~\ref{prop::gromada_tweaked} can be recovered as follows:

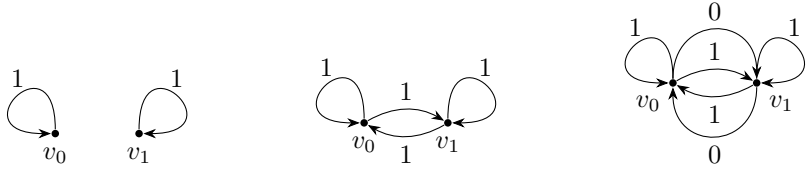
\begin{figure}[htb]
        \begin{tikzpicture}[
            node distance=1cm,
            vertex/.style={circle, fill, inner sep=1pt},
            >=Stealth
          ]
          \begin{scope}
            \node[vertex, label=below:$v_0$] (a0) at (0,0) {};
            \node[vertex, label=below:$v_1$] (b0) [right=of a0] {};
            \path[->] (a0) edge[loop, in=180, out=90, looseness=40] node[above]{$1$} (a0);
            \path[->] (b0) edge [loop, in=0, out=90, looseness=40] node[above]{$1$} (b0);
          \end{scope}
\end{tikzpicture}\quad 
\begin{tikzpicture}[
            node distance=1cm,
            vertex/.style={circle, fill, inner sep=1pt},
            >=Stealth
          ]
          \begin{scope}
            \node[vertex, label=below:$v_0$] (a0) at (0,0) {};
            \node[vertex, label=below:$v_1$] (b0) [right=of a0] {};
            \path[->] (a0) edge[loop, in=180, out=90, looseness=40] node[above]{$1$} (a0);
            \path[->] (b0) edge [loop, in=0, out=90, looseness=40] node[above]{$1$} (b0);
            \path[->] (a0) edge [bend left] node[above]{$1$} (b0);
            \path[->] (b0) edge [bend left] node[below]{$1$} (a0);
          \end{scope}
\end{tikzpicture}\quad 
\begin{tikzpicture}[
            node distance=1cm,
            vertex/.style={circle, fill, inner sep=1pt},
            >=Stealth
          ]
          \begin{scope}
            \node[vertex, label=below left:$v_0$] (a0) at (0,0) {};
            \node[vertex, label=below right:$v_1$] (b0) [right=of a0] {};
            \path[->] (a0) edge[loop, in=180, out=90, looseness=40] node[above]{$1$} (a0);
            \path[->] (b0) edge [loop, in=0, out=90, looseness=40] node[above]{$1$} (b0);
            \path[->] (a0) edge [bend left=30] node[above]{$1$} (b0);
            \path[->] (b0) edge [bend left=30] node[below]{$1$} (a0);
            \path[->] (a0) edge [bend left=90, looseness=2] node[above]{$0$} (b0);
            \path[->] (b0) edge [bend left=90, looseness=2] node[below]{$0$} (a0);
          \end{scope}
\end{tikzpicture}
\caption{Voltage graphs leading to quantum graphs on $M_2(\mathbb{C})$} with one, two, and three quantum edges.
        \label{ex::voltage_graphs}
\end{figure}

\begin{enumerate} 
\item Consider the first voltage graph from the left in Figure~\ref{ex::voltage_graphs}. With this labeling, we have 
\[
\tilde{A}=\left(\begin{pmatrix}0 & 0\\ 0 & 0\end{pmatrix},\begin{pmatrix}1 & 0\\ 0 & 1\end{pmatrix}\right).
\]
The quantum adjacency matrix of the derived quantum graph reads
\[
A=m(X_1\otimes E^*_\rtimes \tilde{A}_1E_\rtimes)m^*=\begin{pmatrix}1 & 0 & 0 & 0\\ 0 & -1 & 0 & 0\\ 0 & 0 & -1 & 0\\ 0 & 0 & 0 & 1\end{pmatrix}.
\]
The resulting quantum graph is the quantum graph on $(M_2,2{\rm Tr})$ whose quantum adjacency matrix is $P_3$, see Proposition~\ref{prop::gromada_tweaked}.

\item Consider the second voltage graph from the left in Figure~\ref{ex::voltage_graphs}. With this labeling, the adjacency matrix of the graph decomposes as follows
\[
\tilde{A}=\left(\begin{pmatrix}0 & 0\\ 0 & 0\end{pmatrix},\begin{pmatrix}1 & 1\\ 1 & 1\end{pmatrix}\right).
\]
The adjacency matrix of the derived quantum graph reads
\[
A=m(X_1\otimes E^*_\rtimes \tilde{A}_1E_\rtimes)m^*=\begin{pmatrix}1 & 0 & 0 & 1\\ 0 & -1 & -1 & 0\\ 0 & -1 & -1 & 0\\ 1 & 0 & 0 & 1\end{pmatrix}.
\]
The resulting quantum graph is the quantum graph on $(M_2, 2{\rm Tr})$ with the adjacency matrix $P_3+P_2$, see Proposition~\ref{prop::gromada_tweaked}. 

\item Consider the third voltage graph from the left in Figure~\ref{ex::voltage_graphs}. With this labeling, the adjacency matrix of this graph decomposes as follows
\[
\tilde{A}=\left(\begin{pmatrix}0 & 1\\ 1 & 0\end{pmatrix},\begin{pmatrix}1 & 1\\ 1 & 1\end{pmatrix}\right).
\]
Note that in this example the $\tilde{A}_0$ component is non-zero. The adjacency matrix of the derived quantum graph reads
\[
A=m(X_0\otimes E^*_\rtimes\tilde{A}_0E_\rtimes+X_1\otimes E^*_\rtimes \tilde{A}_1E_\rtimes)m^*=\begin{pmatrix}1 & 0 & 0 & 2\\ 0 & -1 & 0 & 0\\ 0 & 0 & -1 & 0\\ 2 & 0 & 0 & 1\end{pmatrix}.
\]
The resulting quantum graph is the quantum graph on $(M_2,2{\rm Tr})$ with the adjacency matrix $P_3+P_2+P_1$, see Proposition~\ref{prop::gromada_tweaked}. 
\end{enumerate}
Finally, the last statement of the proposition follows from Theorem~\ref{qiso::thm:quantum_isomorphism_between_derived_graphs}.
\end{proof}

\begin{remark}
Let $(\tilde{\Gamma},\lambda)$ be a pre-simple classical $\mathbb{Z}_2$-voltage graph with two vertices and assume that $(\tilde{\Gamma},\lambda)$ is invariant under the action $\hat{\alpha}$ of $\hat{\mathbb{Z}}_2$ given by swapping the vertices. Then the derived quantum graph $(\tilde{\Gamma},\lambda)\rtimes_{\hat{\alpha}}\hat{\mathbb{Z}}_2$ is isomorphic to an undirected loopfree quantum graph on $(M_2,2{\rm Tr})$. Indeed, proceeding as in the proof of Proposition~\ref{ex::prop:qgraphs_on_M2_as_derived_qgraphs}, one obtains that any $\tilde{A}=(\tilde{A}_0,\tilde{A}_1)$ coming from a pre-simple classical $\mathbb{Z}_2$-voltage graph with respect to the action $\hat{\alpha}$ is of the form
    \[
\tilde{A}_0=\begin{pmatrix}0 & b_0\\ b_0 & 0 \end{pmatrix},\qquad \tilde{A}_1=\begin{pmatrix}a_1 & b_1\\ b_1 & a_1\end{pmatrix},\qquad b_0,a_1,b_1\in\{0,1\}.
\]
Then a straightforward computation yields the adjacency matrix of the corresponding derived quantum graph:
\begin{equation}\label{ex::quantder1}
A=\begin{pmatrix}a_1 & 0 & 0 & a_1\\ 0 & -a_1 & -a_1 & 0 \\ 0 & b_0-b_1 & b_0-b_1 & 0\\ b_0+b_1 & 0 & 0 & b_0+b_1\end{pmatrix}.
\end{equation}
From Proposition~\ref{prop:voltage_to_derived_properties}, we infer that $A$ is undirected and loopfree.
\end{remark}

\printbibliography

\end{document}